\documentclass[a4paper,12pt]{article}
\usepackage[T1]{fontenc}
\usepackage[utf8]{inputenc}
\usepackage[bookmarks,colorlinks=true,linkcolor=blue]{hyperref} 
\usepackage{amsmath,amsthm,amssymb,color,graphicx,bm,mathtools}
\usepackage{layaureo}
\usepackage{microtype}
\usepackage{emptypage}
\usepackage{enumitem}
\usepackage{lmodern}
\usepackage{stmaryrd}
\usepackage[affil-it]{authblk}
\usepackage[english]{babel}
\usepackage{blindtext}
\usepackage{relsize}
\usepackage[all, pdf]{xy}
\usepackage{subfig}
\usepackage{color}
\usepackage{graphicx,booktabs}
\usepackage{lipsum}
\usepackage{tcolorbox}

\DeclarePairedDelimiter{\abs}{\lvert}{\rvert}
\DeclarePairedDelimiter{\norm}{\lVert}{\rVert}

\newcommand{\numberset}{\mathbb}
\newcommand{\N}{\numberset{N}} 
\newcommand{\R}{\numberset{R}} 
\renewcommand{\phi}{\varphi} 
\renewcommand{\chi}{\mathcal{X}}
\newcommand{\C}{\numberset{C}} 
\newcommand{\Z}{\numberset{Z}} 

\renewcommand{\epsilon}{\varepsilon}
\newcommand{\mass}{\mathrm{mass}}
\newcommand{\comass}{\mathrm{com}}
\newcommand{\nucl}{\mathrm{nucl}}

\DeclareMathOperator{\Div}{div}

\newtheorem{theorem}{Theorem}

\newtheorem{lemma}[theorem]{Lemma}

\newtheorem{maintheorem}{Theorem}

\theoremstyle{definition}
\newtheorem{definition}[theorem]{Definition}
\newtheorem{remark}[theorem]{Remark}

{\left\lbrace\begin{aligned}}%
{\end{aligned}\right.}

\usepackage{esint}

\renewcommand{\d}{\mathrm{d}}
\newcommand{\D}{\mathrm{D}}
\renewcommand{\u}{\mathbf{u}}
\newcommand{\ttheta}{{\boldsymbol{\theta}}}
\newcommand{\mres}{\mathbin{\vrule height 1.6ex depth 0pt width
0.13ex\vrule height 0.13ex depth 0pt width 1.3ex}}

\let\div\relax
\DeclareMathOperator{\div}{div}
\DeclareMathOperator{\BV}{BV}
\DeclareMathOperator{\SBV}{SBV}
\DeclareMathOperator{\spt}{spt}
\DeclareMathOperator{\supp}{supp}

\begin{document}
\title{Energy-minimizing torus-valued maps with prescribed singularities, Plateau's problem, and BV-lifting}
\author{Giacomo Canevari and Van Phu Cuong Le}


\date{\today}

\maketitle

\begin{abstract}
In this paper, we investigate the relation between energy-minimizing
torus-valued maps with prescribed singularities,
the lifting problem for torus-valued maps in the space~BV, 
and Plateau's problem for vectorial currents, in codimension one. 
First, we show that the infimum of the $W^{1,1}$-seminorm
among all maps with values in the~$k$-dimensional flat torus
and prescribed topological singularities~$S$ is equal to the 
minimum of the mass among all \textit{normal} $\R^k$-currents, 
of codimension one, bounded by~$S$. 
Then, we show that the minimum of the $\BV$-energy
among all liftings of a given torus-valued $W^{1,1}$-map~$\u$
can be expressed in terms of the minimum mass among all
\textit{integral} $\Z^k$-currents, of codimension one,
bounded by the singularities of~$\u$.
As a byproduct of our analysis, we provide a bound
for the solution of the integral Plateau problem,
in codimension one, in terms of Plateau's problem for normal currents.

\end{abstract}
\section{Introduction}
In the seminal paper of Brezis, Coron, and Lieb~\cite{BrezisCoronLieb}, the authors showed the equivalence between the energy of harmonic maps with prescribed singularities and \textit{minimal connections} between those points. More precisely, in the ambient space $\R^d$, $d\geq 2$, 
given distinct points $P_{1},\ldots,P_{k}$ and $N_{1},\ldots,N_{k}$
(which stand for the topological singularities),
we consider the set of sphere-valued maps
\[
 V:=\left\lbrace \u\in C^{1}(\R^{d} \setminus \lbrace P_{1},\ldots,P_{k}, N_{1},\ldots, N_{k} \rbrace \, , \mathbb{S}^{d-1}), \, \deg(\u,P_{i})=1, \, \deg(\u,N_{i})=-1\right\rbrace \! ,
\]
where~$\deg(\u, P)$ denotes the degree of~$\u$ on a small 
sphere around the point~$P$.
Then, there holds
\begin{equation} \label{min_conn} 
 \inf \left\{ \int_{\R^{d}} |\nabla \u|^{d-1} \, \d x \mid \u \in V \right\}=(d-1)^{\frac{d-1}{2}}\alpha_{d-1}L, 
\end{equation}
where~$\alpha_{d-1}$ is the surface area of the unit sphere~$\mathbb{S}^{d-1}$ in $\R^d$ and~$L$ is the \textit{minimal connection} between $P_{1},\ldots,P_{k}$, and $N_{1},\ldots,N_{k}$ i.e. the solution of 
a Monge-Kantorovich problem in which $P_{1},\ldots,P_{k}$,
and $N_{1},\ldots,N_{k}$ play the role of marginals. 
More precisely, let $S_{k}$ be the set of all the permutations of~$\lbrace 1,\ldots, k \rbrace$. Then, $L$ is defined as
\begin{equation}\label{minimialconnection}
L:=\min_{\sigma \in S_k}\sum_{i=1}^k \abs{P_i-N_{\sigma(i)}}.
\end{equation}
The equality~\eqref{min_conn} carries over, 
with few modifications, to the case the points~$P_{1},\ldots,P_{k}$,
$N_{1},\ldots,N_{k}$ are not distinct; 
see~\cite{BrezisCoronLieb} for the details.
This result was later recast and generalised by Almgren, Browder, and Lieb~\cite{abh}, who interpreted the minimal connection~$L$ as a solution of \textit{Plateau}'s problem --- that is, a minimizer of the mass among all $1$-dimensional 
integral currents~$T$ whose boundary is defined
by the point singularities at~$P_i$, $N_i$:
\begin{equation} \label{Plateau1} 
L=\inf \left\{ \mathbb{M}(T) \mid \partial T=\sum_{i=1}^k\left([P_i]-[N_i]\right) \right\} \! .
\end{equation}
Here, $\mathbb{M}(T)$ denotes the mass of the current~$T$
and~$\partial T$ its boundary. 

Following~\cite{BrezisCoronLieb}, Baldo et al.~recently 
provided~\cite{BaLe} a connection between energy-minimizing maps 
with values in a Cartesian product of spheres 
($\mathbb{S}^{d-1}\times \mathbb{S}^{d-1} 
\times \ldots \times \mathbb{S}^{d-1}, k$-times),
having prescribed topological
singularities at a finite number of points,
and a branched transportation problem i.e.~a
\textit{Gilbert-Steiner} problem with the singular points as sources.
The analysis in~\cite{BaLe} is based on
a reformulation of the branched transportation problem
as a Plateau's problem for $1$-dimensional integral 
currents with coefficients in the normed group~$(\Z^k, \| \cdot \|_{p})$, 
where~$1 \leq p \leq \infty$ and~$\| \cdot \|_{p}$ is the $\ell_p$-norm
on~$\Z^k$, as introduced in~\cite{MaMa, MaMa2}.
More precisely, given points~$P_{1},\ldots,P_{k}, P_{k+1}$ in~$\R^{d}$
with~$k\geq 1$, Baldo et al.~considered
\begin{equation} \label{M} 
 M_p := \inf \left\{ \mathbb{M}_p(T):\,\partial T = e_{1}\delta_{P_1}+e_{2}\delta_{P_2}+\ldots+e_{k}\delta_{P_{k}}-(e_{1}+e_{2}+\ldots+e_{k})\delta_{P_{k+1}} \right\} \! ,
\end{equation}
where~$T$ is a 1-dimensional current with coefficients in the normed group
$(\Z^k, \| \cdot \|_{p})$ and~$\mathbb{M}_p(T)$ is the $p$-mass of~$T$
(we refer the reader to~\cite{MaMa, MaMa2} and Section~\ref{section2}
below for more details).
Now, let~$O_{i}$ be the set of maps
$u\in W^{1,d-1}_{\rm loc}(\R^{d}; \mathbb{S}^{d-1})$
that are constant in a neighbourhood of infinity
and whose Jacobian determinant is given by
$J u = \frac{\alpha_{d-1}}{d}( \delta_{P_i}-\delta_{P_{k+1}})$
in the sense of distributions (see for instance \cite{JeSo02, ABO1}).
In \cite{BaLe}, the authors considered a class of functionals~$E_p$,
defined on $O_1\times \ldots \times O_k$, which depend on~$p$ and 
are comparable to the~$L^{d-1}$-seminorm of the gradient
(more precisely, the energy functional~$E_p$
must satisfy the requirements of \cite[Definition~13]{BaLe}).
Assuming a minimizer of~\eqref{M}
admits a calibration (see Definition~\ref{Calibration}),
they showed that 
\begin{equation}\label{harmonic}
\inf_{O_1\times \ldots \times O_k} E_p
= (d-1)^{\frac{d-1}{2}}\alpha_{d-1} M_p.
\end{equation}
Motivated by \cite{BrezisCoronLieb, BaLe},
in this work we investigate the connections
between energy-minimizing maps with values
in the $k$-dimensional \textit{flat torus} 
($\mathbb{T}^k := \mathbb{S}^1 \times \ldots \times \mathbb{S}^1$, $k$-times)
and Plateau's problems for currents in codimension $1$,
with coefficients in~$(\Z^k, \, \|\cdot\|_p)$
or~$(\R^k, \, \|\cdot\|_p)$.
It is worth mentioning that the study of maps with values in the flat torus arises naturally in several physical contexts. 
For instance, $\mathbb{T}^2$-valued maps are considered
in two-component Ginzburg-Landau models,
in which the order parameter is a complex vector-valued map
accounting for ferromagnetic or antiferromagnetic effects 
in high-temperature superconductors (see e.g.~\cite{Stan} 
and the references therein).

Another motivation for our work is the study of
Plateau's problem for vectorial currents, in codimension one.
More precisely, we would like to compare the minimal mass between Plateau's problem for \textit{normal currents} and \textit{integral currents},
with coefficients in the normed Abelian groups $(\R^k, \| \, \|_{p})$, $(\Z^k, \| \, \|_{p})$ respectively. 
It is well-known~\cite[5.10]{Federer2} that in case $k=1$, i.e.~for classical currents, given a $(d-2)$-dimensional integral flat boundary~$S$ in $\R^d$ with coefficients in $\mathbb{Z}$, we have the following equality:
\begin{equation}\label{equalifyFederer} 
\mathbb{P}_{\Z}(S)=\mathbb{P}_{\R}(S),
\end{equation}
where $\mathbb{P}_{\Z}(S)$ is the 
minimal mass among all the integral currents whose boundary is $S$
and~$\mathbb{P}_{\R}(S)$ is the 
minimal mass among all the normal currents whose boundary is $S$.

Then, it is natural to ask whether the equality~\eqref{equalifyFederer}
extends to vectorial currents, with~$k>1$. More precisely,
given a $(d-2)$-dimensional integral boundary~$S$ in $\R^d$
with coefficients in $\mathbb{Z}^k$, let
\begin{equation} \label{Plateau_int} \tag{P$_{\Z^k}$}
    \begin{split}
	\mathbb{P}_{\Z^k,p}(S) := \inf\big\{\mathbb{M}_p(T) 
	\colon &T \textrm{ is an } \textit{integral } (d-1)\textrm{-current} \\
	&\textrm{with coefficients in }
	(\Z^k, \, \norm{\cdot}_p), \ \partial T = S \big\}
    \end{split}
\end{equation}
and let
\begin{equation} \label{Plateau_normal} \tag{P$_{\R^k}$}
	\begin{split}
	\mathbb{P}_{\R^k,p}(S) := \inf\big\{\mathbb{M}_p(T) 
	\colon &T \textrm{ is a } \textit{normal } (d-1)\textrm{-current} \\
	&\textrm{with coefficients in }
	(\R^k, \, \norm{\cdot}_p), \ \partial T = S \big\}.
	\end{split}
\end{equation}
Here, $\mathbb{M}_p$ denotes the mass of a current,
taken with respect to the~$\ell^p$ norm~$\norm{\cdot}_p$.
The definition immediately implies~$\mathbb{P}_{\Z^k,p}(S)
\geq\mathbb{P}_{\R^k,p}(S)$. However, the equality does \emph{not} 
hold true, in general: in~\cite[Example 4.2]{Bonafini2018},
the authors gave an example (with~$d=2$, $k=4$, $p=\infty$) where
\begin{equation}
\mathbb{P}_{\Z^k,\infty}(S)> \mathbb{P}_{\R^k,\infty}(S).
\end{equation}
(see also Remark~\ref{rk:noncalibration} below for more details).


Instead, one could try to derive a bound on~$\mathbb{P}_{\Z^k,p}(S)$ in terms of~$\mathbb{P}_{\R^k,p}(S)$. We observe that we always have the following inequality:
\begin{equation}\label{inequalityPlateau1}
\mathbb{P}_{\Z^k,p}(S) \leq k \, \mathbb{P}_{\R^k,p}(S).
\end{equation}
To see this,
let~$(S_1, \, \ldots, \, S_k)$ be the components of~$S$.
For any~$i$, let $T_i$ be an integral current of minimal mass among all integral currents bounded by~$S_i$.
Then, on the one hand, $(T_1, \, \ldots, \, T_k)$
is an admissible competitor for problem~$\mathbb{P}_{\Z^k,p}(S)$, so
$\mathbb{P}_{\Z^k,p}(S) \leq \sum_{i=1}^k \mathbb{M}(T_i)$.
On the other hand, Federer's result~\cite[5.10]{Federer2} implies that~$T_i$
also minimizes the mass among the normal currents
bounded by~$S_i$. Therefore, we have 
$\mathbb{M}(T_i) = \mathbb{P}_{\R}(S_i) \leq \mathbb{P}_{\R^k,p}(S)$
for any~$i$ and any~$p$, and~\eqref{inequalityPlateau1} follows. 
In Theorem~\ref{E} below, we provide a different bound, in terms of a constant that grows \emph{sublinearly} as a function of~$k$
(when~$p$ is finite). More precisely, we prove that
\begin{equation}\label{inequalityPlateau2}
\mathbb{P}_{\Z^k,p}(S)\leq (2k^{1-\frac{1}{p}}-1)
\, \mathbb{P}_{\R^k,p}(S).
\end{equation}
The constant factor in front of the right-hand side
is independent on the dimension of the ambient space~$\R^d$:
it depends only on the number of components~$k$ and on~$p$.
We observe that the inequality \eqref{inequalityPlateau2}
is stronger than \eqref{inequalityPlateau1} when~$p < \infty$
and~$k$ is large enough --- for instance, when $1 \leq p \leq 2$ and $k>1$
(because then~$2 k^{1 - 1/p} - 1 \leq 2 k^{1/2} - 1 < k$).
Moreover, the inequality~\eqref{inequalityPlateau2} 
is sharp when~$k = 1$ or~$p = 1$ (for in this case, 
it reduces to~$\mathbb{P}_{\Z^k,p}(S) = \mathbb{P}_{\R^k,p}(S)$),
but it is not sharp when~$k>1$ and~$p=\infty$
(for in this case, the estimate~\eqref{inequalityPlateau1} 
is stronger than~\eqref{inequalityPlateau2}).
We do not know whether~\eqref{inequalityPlateau1}
is sharp for~$k>1$, $1 < p < \infty$.
In fact, to the best of our knowledge, we are not aware
of any inequality that provides an estimate for~$\mathbb{P}_{\Z^k,p}(S)$
in terms of~$\mathbb{P}_{\R^k,p}(S)$, other 
than~\eqref{inequalityPlateau1} and~\eqref{inequalityPlateau2}.

\paragraph*{Main results.}
Throughout the paper, we will consider currents 
in~$\R^d$ with coefficients in 
the normed space~$(\R^k, \, \norm{\cdot}_p)$,
where~$p\in [1, \, \infty]$, as defined by 
Marchese and Massaccesi~\cite{MaMa}.
(For convenience of the reader, we recall the 
definition in Section~\ref{section2} below).
We consider a~$(d-2)$-dimensional current~$S$,
with coefficients in~$\mathbb{Z}^k$. We will always assume that $S$ is an
$\mathbb{Z}^k$-integral flat boundary ---
that is, there exists a $(d-1)$-dimensional, rectifiable current~$T$,
with multiplicity in~$\mathbb{Z}^k$, such that~$\partial T = S$ ---
and that~$S$ has compact support. Moreover, we will often
make a (rather mild) assumption on the support of $S$. More precisely, if~$(S_1, \, \ldots, \, S_k)$ are the components 
of~$S$, we require that 
\begin{equation} \label{hp:H}
 \textrm{for each } i=1, \, \ldots, \, k, 
 \quad \R^d\setminus\spt S_i \ \textrm{ is connected and }
 \mathcal{H}^{d-1}(\spt S_i) = 0,
\end{equation}
where~$\spt S_i$ denotes the support of~$S_i$
and~$\mathcal{H}^{d-1}$ the Hausdorff measure of dimension~$d-1$.
However, we do not need to assume the condition~\eqref{hp:H} in Theorem \ref{E}.
Let~$(S_1, \, \ldots, \, S_k)$ be the components of~$S$.
For each index~$i=1, \, \ldots, \, k$, we consider the space~$G_i$, defined as the subset of $W^{1,1}_{\rm loc}(\R^{d}; \mathbb{S}^{1})$ consisting of the functions $u_{i}$ that are constant outside an open ball
of radius $r_i = r_{i}(u_i)$ 
and are such that $\star J(u_i)=\pi S_{i}$.
(Here, $J(u_i)$ denotes the distributional Jacobian of~$u_i$
and~$\star$ is the Hodge star operator; see Section~\ref{section2}
for details.)
We define $\mathcal{Q}_{S}=G_{1}\times G_{2} \times \ldots \times G_{k}$. 
As it turns out, if~$S$ has finite mass and
the assumption~\eqref{hp:H} is satisfied, then
the set $\mathcal{Q}_{S}$ is nonempty --- see Remark~\ref{nonempty}.

Let~$p\in [1, \, \infty]$ be fixed.
In this paper, we shall consider two functionals.
The first one is the $W^{1,1}$-harmonic energy, defined by
\begin{equation}
H_{p}(\u):=\int_{\R^d}\|(|\nabla u_1|,\ldots,|\nabla u_k|)\|_p \, \d x
\end{equation}
for any~$\u=(u_{1},\ldots,u_{k})\in \mathcal{Q}_{S}$.
Correspondingly, we set
\begin{equation} \label{H} \tag{H}
\mathbb{H}_{p}(S):=\inf \left\{H_{p}(\u)\colon \u
\in \mathcal{Q}_{S} \right\} \! .
\end{equation}
The second functional we consider is given by
\begin{equation}\label{FH21}
E_{p}(\u):=\int_{\R^{d}} |\nabla \u (x)|_{\nucl,p}\,\d x
\end{equation}
for any~$\u\in \mathcal{Q}_{S}$, where $| \cdot |_{\nucl,p}$ 
is the so-called $p$-nuclear norm. This is a suitable norm on the space of 
linear operators~$\R^d\to\C^k$, defined in \eqref{normonthegradient}.
We set
\begin{equation}\label{N}\tag{E}
\mathbb{E}_{p}(S):=\inf \left\{ E_{p}(\u)\colon \u
\in \mathcal{Q}_{S} \right\} \! .
\end{equation}
The functional~$H_p$ has the advantage of being rather simple to compute, 
because its integrand is defined component-wise.
This will make it a handy tool in our analysis of the lifting problem, 
in the second part of the paper (see Section~\ref{section4} below).
On the other hand, the functional~$E_p$ turns to be quite a natural object
to consider, because the nuclear norm of the gradient~$\abs{\nabla\u}_{\nucl,p}$
equals the mass of a suitable vector field, depending on~$\u$ 
(see Lemma~\ref{equivalentmassnormandgradient} below). In particular,
$E_p$ ``is suitably related to~$\mathbb{M}$ and~$H_p$'', in the sense
of~\cite[Definition~13]{BaLe}.
In the first part of the paper, we investigate the relationship
between Problems~\eqref{H}, \eqref{N} and~$\mathbb{P}_{\R^k,p}(S)$.
\begin{maintheorem}\label{A} 
	Let $S$ be a $\Z^k$-integral flat boundary,
	of dimension~$(d-2)$. Assume that~$S$ has
	finite mass and compact support
	and the condition~\eqref{hp:H} is satisfied.
    Then, for any $p\in [1, \, \infty]$, one has
	\begin{equation}\label{equalityenergymimizingmapsandplateus} 
	\mathbb{H}_{p}(S)\leq \mathbb{E}_{p}(S) = 2\pi \, \mathbb{P}_{\R^k,p}(S) \leq \mathbb{H}_{1}(S)\leq k^{1-\frac{1}{p}}\mathbb{H}_{p}(S).
	\end{equation}
\end{maintheorem}
\begin{remark}
	In the planar case, i.e.~when~$d=2$, the first part of present paper is in the same spirit as~\cite{BaLe}. In particular, our functional~$\mathbb{E}$ satisfies the requirements of \cite[Definition 13]{BaLe}. Furthermore, we are able to prove the equivalence between the problems~\eqref{N} and~$\mathbb{P}_{\R^k,p}(S)$  thanks to the special structure of the target manifold. It turns out that the main result of~\cite{BaLe} is not always true in case without assuming the existence of calibration: in Remark \ref{rk:noncalibration}, by combining Theorem~\ref{A} with 
	Example 4.2 from~\cite{Bonafini2018}, we provide a counterexample for the equality~\eqref{harmonic} in case $p=\infty$, $k=4$, and $d=2$.
	Instead, in Theorem~\ref{D} below, we give a necessary and sufficient 
	condition for the equality 
	$\mathbb{H}_p(S) = 2\pi \,\mathbb{P}_{\Z^k,p}(S)$
	to hold, in terms of the lifting problem for torus-valued maps.
\end{remark}

In the second part of the paper,
we investigate the relationship between Plateau's problem
and the \textit{lifting} problems for torus-valued maps of bounded variation. 
 Given~$\u\in\mathcal{Q}_S$, we will
say that~$\ttheta=(\theta_{1},\ldots,\theta_{k})
\in \BV_{\mathrm{loc}}(\R^d, \R^k)$ is a BV-lifting
of~$\u$ if~$u_j = e^{i\theta_j}$ a.e.~in~$\R^d$,
for any index~$j=1, \, \ldots, \, k$.
Any map~$\u\in\mathcal{Q}_S$ admits a BV-lifting
\cite{GiaquintaModicaSoucek, Ignat, Ignat-Lifting}.
Now, given~$u\in\mathcal{Q}_S$, we consider the following problem:
\begin{equation} \label{min_lifting} \tag{L}
\mathbb{L}_p(S, \, \u) := \inf\left\{ \int_{\mathrm{S}(\ttheta)} 
\|\ttheta^+(x) - \ttheta^-(x)\|_p \, \d \mathcal{H}^{d-1}(x) \colon
\ttheta \textrm{ is a BV-lifting of } \u \right\} \! ,
\end{equation}
where~$\mathrm{S}(\ttheta)$ is the $(d-1)$-rectifiable jump set of~$\ttheta$, and $\ttheta^+(x)$, $\ttheta^-(x)$ stand for the right and left approximate limits of $\ttheta$ at almost every point~$x\in \mathrm{S}(\ttheta)$ 
(we refer the reader to Section \ref{section4} for more details).
As we will prove later on (in Lemma~\ref{lemma:indip_S}),
the quantity~$\mathbb{L}_p(S, \, \u)$ 
actually depends on~$S$, but not on the choice of~$\u\in\mathcal{Q}_S$,
so we can write~$\mathbb{L}_p(S)$ instead of~$\mathbb{L}_p(S, \, \u)$.
The relation between the problems~\eqref{min_lifting}
and~\eqref{Plateau_int} is encoded in the following theorem:
\begin{maintheorem}\label{C}
	Let $S$ be a $\Z^k$-integral flat boundary,
	of dimension~$(d-2)$. Assume that~$S$ has
	finite mass and compact support
	and the condition~\eqref{hp:H} is satisfied. Then, 
	for any~$p\in [1, \, \infty]$ 
	one has
	\begin{equation}
	\mathbb{L}_p(S)=2\pi \,\mathbb{P}_{\Z^k,p}(S) .
	\end{equation}
\end{maintheorem}

Finally, combining Theorem~\ref{A} and Theorem~\ref{C}
with suitable bounds on the norm of BV-liftings
(see~\cite{Ignat}), we obtain a bound for~$\mathbb{P}_{\Z^k,p}(S)$ 
in terms of~$\mathbb{P}_{\R^k,p}(S)$.
\begin{maintheorem}\label{E}
Let $S$ be a $\Z^k$-integral flat boundary,
of dimension~$(d-2)$, of compact support.
Then, for all~$p\in [1, \, \infty]$, we have
\begin{equation} \label{Plateau_ineq}
	\mathbb{P}_{\Z^k,p}(S)\leq (2k^{1-\frac{1}{p}}-1)
	\, \mathbb{P}_{\R^k,p}(S).
\end{equation}
\end{maintheorem}

Note that Theorem~\ref{E} does \emph{not} rely on the assumption
that~$S$ has finite mass, nor that it satisfies~\eqref{hp:H}.
We have removed these assumptions by an approximation procedure.

\begin{remark}
	In case $k=1$, when we are working with~$\mathbb{S}^{1}$-valued maps,
	the quantities defined in~\eqref{Plateau_int},
	\eqref{Plateau_normal}, \eqref{min_lifting}, \eqref{H} and~\eqref{N}
	do not depend on~$p$. Moreover, from Theorem~\ref{A}, Theorem~\ref{C}
	and Theorem~\ref{E} we deduce (dropping the dependence of~$p$ 
	from the notation) that
	\begin{equation} \label{S1equality}
	 \mathbb{H}(S) 
	 = \mathbb{E}(S)
	 = \mathbb{L}(S)
	 = 2\pi \, \mathbb{P}_{\R}(S) 
	 = 2\pi \, \mathbb{P}_{\Z}(S) .
	\end{equation}
	The equality~$\mathbb{H}(S) = 2\pi \, \mathbb{P}_{\Z}(S)$
	was first conjectured by Brezis, Coron and
	Lieb~\cite[Equation~(8.23)]{BrezisCoronLieb}. 
	The same authors gave a proof of this equality
	in a particular case, namely, when~$d=3$ and 
	the boundary~$S$ is a planar curve.
	Later on, Almgren, Browder and Lieb~\cite{abh} proved 
	it in greater generality, using the coarea formula
	(see also the recent book by Brezis and Mironescu~\cite{Brezis2021},
	in particular Chapters~2 and~3). In our paper, 
	we provide an alternative proof, which is based
	on the existence of BV-liftings instead of the coarea formula.
\end{remark}
\begin{remark}
  We stress that our arguments do \emph{not}
 depend on Federer's result~\cite{Federer2} 
 that~$\mathbb{P}_{\Z}(S) = \mathbb{P}_{\R}(S)$ in case~$k=1$.
 In fact, Theorem~\ref{E} provides an alternative proof for this fact.
\end{remark}

The paper is organized as follows. In Section~\ref{section2}, we recall 
basic notions in Geometric Measure Theory which will be used in the paper. 
Section~\ref{section3} is devoted to the proof Theorem~\ref{A}.
In Section~\ref{section4} we address the lifting problem for
torus-valued maps in the space~$\BV$, 
and we prove Theorems~\ref{C} and~\ref{E}.
In addition, still in Section~\ref{section4},
we provide a necessary and sufficient condition for the
problems~\eqref{H} and~\eqref{Plateau_int} to be equivalent.
An appendix, containing the proof of some technical results, 
completes the paper.
\paragraph*{Acknowledgements.}
The authors are grateful to Andrea Marchese for fruitful discussions
and to the referees for their careful reading of the manuscript
and their helpful suggestions.
The authors acknowledge support from the University of Verona,
under the project RIBA~2019, No.~RBVR199YFL
``Geometric Evolution of Multi Agent Systems''.
The authors were partially supported by GNAMPA-INdAM.
The second author was partially supported by the German Research Foundation (DFG) by the project no. 392124319, and also acknowledges the support of STRUCTURES, Heidelberg University.
\section{Notations and preliminaries}
\label{section2}

\numberwithin{equation}{section}
\numberwithin{theorem}{section}

\subsection{Preliminary notions}
\subsubsection{Currents with coefficients in normed groups and Plateau's problems}
\label{section2.1.1}

Throughout this entire paper, we work in the ambient space~$\R^{d}$,
with~$d\geq 2$.
We denote by~$\mathbb{B}^d_r(x)$ the open ball in $\R^d$ with center at~$x\in \R^d$ and radius~$r>0$, and by~$\bar{\mathbb{B}}^d_r(x)$ its closure.
For each integer~$k>0$, $\mathbb{S}^{k-1}=\partial \mathbb{B}^k_1(0)$ is the unit sphere in $\R^k$ and $\alpha_{k}=\mathcal{H}^{k-1}(\mathbb{S}^{k-1})$ is its surface area, where~$\mathcal{H}^k$ is the $k$-dimensional Hausdorff measure.
We denote by~$\mathcal{L}^{k}$ the Lebesgue measure in~$\R^k$, and we use the notation $\abs{E}$ to denote the Lebesgue measure of the measurable set $E$ for convenience.

We follow~\cite{Ma} for basic notions and terminologies regarding currents with coefficients in a normed group.
Let~$k\geq 1$ be an integer and let~$1 \leq p \leq \infty$.
The normed groups that we consider in this paper are $(\R^k, \| \cdot \|_{p})$ (corresponding to normal currents) and $(\Z^k, \| \cdot \|_{p})$ (corresponding to integral currents),
where $\| \cdot \|_{p}$ is the $\ell_p$-norm in $\R^k$,
defined for $z = (z_1, \ldots, \, z_k)\in\R^k$ by
\begin{equation}\label{norm}
\|z\|_{p}
:=\begin{cases} 
\left(\sum_{j=1}^{k}|z_{j}|^{p}\right)^{\frac{1}{p}} & \mbox{in case } p \in [1, \infty), \\ 
\max \lbrace \abs{z_{1}},\ldots,\abs{z_{k}} \rbrace & \mbox{in case } p=\infty. 
\end{cases}
\end{equation}
For the convenience of the reader, we recall here the main definitions.

Given an integer~$m$ with~$1\leq m \leq d$, we denote by~$\Lambda_{m}(\R^{d})$ the space of $m$-vectors and  by~$\Lambda^{m}(\R^{d})$ the space of $m$-covectors in $\R^{d}$.
Let $\lbrace e_{1},e_{2},\ldots,e_{d} \rbrace$  be an orthonormal basis of $\R^{d}$, and let $\lbrace \d x_1, \d x_2,\ldots,\d x_d \rbrace$ be its dual basis.
Given a multi-index $I=\lbrace i_1,\ldots,i_m \rbrace$ such that $i_1<i_2<\ldots<i_m$, we use the notation $e_{I} := e_{i_1}\wedge \ldots \wedge e_{i_m}$ and $\d x_{I} := \d x_{i_1}\wedge \ldots \wedge \d x_{i_m}$, where $\wedge$ is the wedge product. Then, we introduce the following operators:
\begin{equation}\label{identitication1}
\begin{aligned}
\# \colon \Lambda^{m}(\R^{d}) \to \Lambda_{m}(\R^{d}),
\qquad \flat\colon \Lambda_{m}(\R^{d}) \to \Lambda^{m}(\R^{d}).
\end{aligned}
\end{equation}
The operator~$\#$ is the identification between $m$-covectors and $m$-vectors induced by Euclidean metric i.e. for any $m$-covector~$w=\Sigma_{I}a_{I}\d x_{I}$, the corresponding vector $w^{\#}$ is defined by
\begin{equation}
w^{\#} := \Sigma_{I}a_{I} \, e_{I},
\end{equation}
where the sum is taken over all multi-indices~$I=\lbrace i_1,\ldots,i_m \rbrace$ such that $i_1<i_2<\ldots<i_m$. The operator~$\flat$
is the inverse of~$\#$:
for any $m$-vector~$w=\Sigma_{I}a_{I}e_{I}$, the corresponding covector~$w^{\flat}$ is defined by
\begin{equation}
w^{\flat} := \Sigma_{I}a_{I} \, \d x_{I},
\end{equation}
where again the sum is taken over all multi-indices~$I=\lbrace i_1,\ldots,i_m \rbrace$ such that $i_1<i_2<\ldots<i_m$.

\begin{definition}\label{definitionofcovectornorm}
 An $(\R^{k})^{*}$-valued  $m$-covector on $\R^{d}$ is a bilinear map
 $$w\colon\Lambda_{m}(\R^{d})\times \R^{k}\longrightarrow \R\,.$$
 An $\R^{k}$-valued  $m$-vector on $\R^{d}$ is a bilinear map
 $$w\colon\Lambda^{m}(\R^{d})\times (\R^{k})^*\longrightarrow \R\,.$$
\end{definition}

We denote by $\Lambda_{\R^k}^m(\R^d)$ and $\Lambda_{m,\R^k}(\R^d)$ the spaces of $(\R^{k})^{*}$-valued $m$-covectors and $(\R^{k})$-valued $m$-vectors on $\R^{d}$, respectively. These spaces are dual each other. We observe that $(\R^k)^*$-valued covectors and $\R^k$-vectors can be identified
with $k$-tuples of classical covectors and vectors. More precisely, let $\lbrace e_{1},e_{2},\ldots,e_{k} \rbrace$  be an orthonormal basis of~$\R^{k}$, and let $\lbrace e^{*}_{1}, e^{*}_{2}, \ldots, e^{*}_{k} \rbrace$ be its dual. Then, each $(\R^{k})^{*}$-valued $m$-covector on $\R^{d}$ can be expressed as
$w=w_{1} \, e^{*}_{1}+\ldots+w_{k} \, e^{*}_{k}$,
where for any index~$i= 1, \ldots, k$, $w_{i}$ is a classical $m$-covector in $\R^{d}$, defined by $w_{i}:=\langle w ; \cdot , e_{i} \rangle$. Similarly, each $\R^{k}$-valued $m$-vector on $\R^{d}$ can be expressed as
$v=v_{1} \, e_{1}+\ldots+v_{k} \, e_{k}$ where  $v_{i}$ is a classical $m$-vector in $\R^{d}$ defined by $v_{i}:=\langle v ; \cdot , e^*_{i} \rangle$. The operators $\#$, $\flat$ can be extended 
to operators~$\Lambda^{m}_{\R^{k}}(\R^{d})\to\Lambda_{m, \R^{k}}(\R^{d})$ 
and $\Lambda_{m, \R^{k}}(\R^{d})\to\Lambda^{m}_{\R^{k}}(\R^{d})$ respectively,  by applying~$\#$, $\flat$ component-wise.

Given~$p\in [1, \, \infty]$, the space $\Lambda_{\R^k}^m(\R^d)$ is equipped with the $p$-\textit{comass} norm, defined by
\begin{equation} \label{comass}
\abs{w}_{\comass, p}:=\sup \left\lbrace \| \langle w ; \tau, \cdot \rangle\|^{*}_{p} 
\colon \vert \tau \vert \leq 1, \ \tau \mbox{ is simple $m$-vector}\right\rbrace\!
\end{equation}
for any $w\in \Lambda^{m}_{\R^{k}}(\R^{d})$.
Here, $|\cdot |$ is the Euclidean norm on $\Lambda^{m}(\R^{d})$, $\Lambda_{m}(\R^{d})$. We identify the dual space $((\R^{k})^{*}, \|\cdot\|^{*}_{p})$
with~$(\R^{k}, \| \cdot\|_{p^{*}})$ where $p^{*}\in [1, \, \infty]$ is the conjugate exponent of $p$, defined in such a way that $\frac{1}{p^{*}}+\frac{1}{p}=1$ (on the understanding that $\frac{1}{\infty}=0$, $\frac{1}{0} = \infty$). 
The space $\Lambda_{m, \R^k}(\R^d)$ is equipped with the pre-dual norm or $p$-\textit{mass} norm, defined as 
\begin{equation}\label{def_massnorm}
|v|_{\mass,p}:=\sup \left\{ \langle w, v \rangle \colon w \in \Lambda^{m}_{\R^{k}}(\R^{d}), \ \abs{w}_{\comass, p}\leq 1 \,  \right\}
\end{equation}
for any~$v\in\Lambda_{m, \R^{k}}(\R^{d})$.
In case~$k=1$, the mass and comass norms are independent of~$p$ 
(i.e.~$\abs{v}_{\mass,p} = \abs{v}_{\mass,1}$, $\abs{w}_{\comass, p} = \abs{w}_{\comass, 1}$
for any~$p\in [1, \, \infty]$, $v\in\Lambda_m(\R^d)$, $w\in\Lambda^m(\R^d)$) 
and reduce to the classical mass and comass. 
In this case, we write~$\abs{v}_{\mass} := \abs{v}_{\mass, 1}$,
$\abs{w}_{\comass} := \abs{w}_{\comass, 1}$ for any~$v\in\Lambda_m(\R^d)$,
$w\in\Lambda^m(\R^d)$.

\begin{remark}
\begin{itemize}
\item[(i)] The mass norm of a vector~$v\in\Lambda_{m, \R^{k}}(\R^{d})$
can be equivalently characterized as
\begin{equation} \label{massnorm-char}
|v|_{\mass,p} 
=\inf \left\{ \sum_{i=1}^l \|z_i\|_{p} \, |v_i| \colon  v=\sum_{i=1}^l z_i \otimes v_{i}, \, z_i \in \R^k, \, v_{i} \mbox{ is a simple $m$-vector in } \R^d \right\}
\end{equation}
(for a proof, see~\cite[\S~1.8.1]{FeBook} for the scalar case~$k=1$
and Lemma~\ref{characterizationmassnorm} in the appendix for the case~$k > 1$).
\item[(ii)] Let $v\in\Lambda_{m,\R^{k}}(\R^{d})$
be a vector with components~$v_1$, \ldots, $v_k\in\Lambda_m(\R^d)$.
If~$1 \leq p < \infty$, then there holds
\begin{equation} \label{comparison-massintro}
\left(\sum_{i=1}^k \abs{v_i}^p\right)^{\frac{1}{p}} \leq \left(\sum_{i=1}^k \abs{v_i}_{\mass}^p\right)^{\frac{1}{p}}
\leq \abs{v}_{\mass,p} \leq \sum_{i=1}^k \abs{v_i}_{\mass}.
\end{equation}
See Lemma \ref{comparenorms} in the appendix for the proof of~\eqref{comparison-massintro},
as well as more general results on the comparison 
between different norms of a given
$\R^k$-valued $m$-vector or $(\R^k)^*$-valued $m$-covector.
\end{itemize} 
\label{remarknotation}
\end{remark}

\begin{definition}
		An $(\R^{k})^{*}$-valued $m$-dimensional differential form,
		or~$(\R^{k})^{*}$-valued $m$-form, defined on $\R^{d}$ is a map 
		$$\omega\colon \R^{d} \longrightarrow \Lambda^{m}_{\R^{k}}(\R^{d})\,.$$
\end{definition}

As before, $\omega$ can be represented as 
$\omega=\omega_{1} e^{*}_{1}+\ldots+\omega_{k}e^{*}_{k}\,,$
and the regularity of $\omega$ is inherited from the
regularity of its components $\omega_{i}$, for $i=1,\ldots,k$. 
Let $\omega=\omega_{1} \, e^{*}_{1}+\ldots+\omega_{k} \, e^{*}_{k}\in C^{1}(\R^d;\Lambda^{m-1}_{\R^{k}}(\R^{d}))$. We denote
$${\rm d}\omega:={\rm d}\omega_{1} \, e^{*}_{1}+\ldots+{\rm d}\omega_{k} \, e^{*}_{k},$$
where ${\rm d}\omega_{i}$ is the differential of $\omega_{i}$.
By definition, we have ${\rm d} \omega\in C(\R^d;\Lambda^{m}_{\R^{k}}(\R^{d}))$.
We define the support of~$\omega$ as $\supp\, \omega := {\rm cl}\lbrace x\in \R^d \colon \omega(x) \neq 0\rbrace$, where ${\rm cl}$ indicates the closure of a set.
We can now define the notion of currents with coefficients in~$\R^k$.
\begin{definition}\label{normalcurrent}
		An $m$-dimensional current, or~$m$-current, $T$ with coefficients in $(\R^{k}, \, \| \cdot \|_{p})$ is a linear and continuous map
		\[
		 T\colon C^{\infty}_{\mathrm{c}}\left(\R^{d};\Lambda^{m}_{\R^{k}}(\R^{d})\right) \longrightarrow \R\,.
		\]
		The continuity of~$T$ is understood with respect to the (locally convex) topology on the space of test forms
		$C^{\infty}_{\mathrm{c}}\left(\R^{d};\Lambda^{m}_{\R^{k}}(\R^{d})\right)$, which is analogous to the standard topology on~$C^\infty_{\mathrm{c}}(\R^d; \R)$.
\end{definition}
 Observe that a vector-valued current
can be written in components, and that we will often
use the notation $T = (T_1, \ldots, T_k)$, 
where $T_i$ is the classical current defined by
\begin{equation}
T_i(\omega) := T(\omega \, e^*_i) 
\qquad \textrm{for any } 
\omega \in C^{\infty}_{\mathrm{c}}\left(\R^{d};\Lambda^{m}(\R^{d})\right).   
\end{equation}
The mass of $T$ is defined as
\[
    \mathbb{M}_p(T):=\sup \left\{ T(\omega):
    \omega \in C^{\infty}_{\mathrm{c}}\left(\R^{d};\Lambda^{m}_{\R^{k}}(\R^{d})\right) \! ,
     \ \|\omega\|_{\comass, p}  \leq 1 \right\}\!,
\]
where~$\|\omega\|_{\comass, p}=\sup_{x\in \R^{d}}|\omega(x)|_{\comass, p}$. When $k=1$, $\mathbb{M}_p$ is equal to the mass of classical currents for any $p\in [1,\infty]$, therefore we shall denote it by $\mathbb{M}$.

Moreover, we define the boundary $\partial T$ of $T$ as
an~$(m-1)$-dimensional current, such that 
\begin{equation} \label{boundary}
 \partial T(\omega):=T(\d\omega)
\end{equation}
for any~$\omega\in C^{\infty}_{\mathrm{c}}(\R^d, \Lambda^{m-1}_{\R^k}(\R^d))$.
The boundary operator acts component-wise.
The space of $\R^k$-currents of dimension $m$ is denoted by $\mathcal{D}^\prime(\R^d, \Lambda_{m,\R^k}(\R^d))=(C^{\infty}_{\mathrm{c}}(\R^d, \Lambda^{m}_{\R^k}(\R^d)))^\prime$.
In case $k=1$, we write~$\mathcal{D}^\prime(\R^d, \Lambda_{m}(\R^d))$ for the space of classical currents.
The support of the current~$T$, denoted by~$\spt(T)$, is defined by:
\begin{equation}
    \spt(T) := \R^d \setminus \cup \left\{ W \colon
    W\mbox{ is open, }\, T(\phi)=0 \mbox{ whenever } \phi \in C^{\infty}_{\mathrm{c}}\left(W;\Lambda^{m}_{\R^{k}}(\R^{d})\right)
    \right\}.
\end{equation}

Weak$^*$ convergence of a sequence of currents,
$T^\ell\rightharpoonup^* T$, is defined by duality with
the space of test forms~$C^\infty_{\mathrm{c}}\left(\R^{d};\Lambda^{m}(\R^{d})\right)$. It turns out that a sequence of $\R^k$-valued currents
converge weakly$^*$ if and only the~$j$-th components of the sequence 
converge weakly$^*$ to the $j$-th component of the limit, for any index~$j$.

\begin{definition}
 A current $T$ is said to be an $\R^k$-normal current
 if $\mathbb{M}_p(T)+\mathbb{M}_p(\partial T)<\infty$.
\end{definition}



\begin{definition}\label{definitionofcurrent}
        A $m$-dimensional rectifiable current with coefficients in $(\Z^{k}, \| \cdot \|_{p})$ (or a $\Z^k$-rectifiable current) is a $m$-dimensional current such that there exists a $m$-dimensional oriented rectifiable set $\Sigma\subset\R^d$, an approximate tangent vectorfield $\tau\colon \Sigma \longrightarrow \Lambda_{m}(\R^{d})$, and a density function $\theta\colon \Sigma \longrightarrow \Z^{k}$ such that 
		\begin{equation} \label{rectifiable_current}
		 T(\omega)=\int_{\Sigma}\langle \omega (x); \tau (x), \theta (x) \rangle \,\d\mathcal{H}^{m}(x)
		\end{equation}
		for every $\omega \in C^{\infty}_{\mathrm{c}}\left(\R^{d};\Lambda^{m}_{\R^{k}}(\R^{d})\right)$. We say that $T$ is an integral current if both $T$ and $\partial T$ are $\Z^k$-rectifiable currents. A $m$-dimensional polyhedral current is a finite union of $m$-dimensional oriented simplexes $\Sigma_i$, each equipped with constant multiplicity~$\sigma_i\in \Z^k$, such that
		for any $i\neq j$, the intersection between $\Sigma_i$ and $\Sigma_j$ is either empty or a common face. In addition, an $m$-dimensional 
		current $T$ is said to be a cycle if $\partial T=0$; $T$ is called a boundary if there exists an~$(m+1)$-dimensional current~$R$ such that~$T = \partial R$.
\end{definition}
We write~$T := \llbracket\Sigma, \tau, \theta\rrbracket$
for the rectifiable current defined in~\eqref{rectifiable_current}.

%
%
%
%

\begin{remark}\label{remarksoncurrents}
	Given an integrable vector field $\tau=(\tau_1,\ldots,\tau_k) \in L^{1}\left(\R^{d};\Lambda_{m,\R^k}(\R^{d})\right)$, integration against~$\tau$
	defines a current $T$ 
	--- that is, each component of~$T = (T_1, \, \ldots, \, T_k)$
	is represented by integration against~$\tau_i\in L^{1}\left(\R^{d};\Lambda_{m}(\R^{d})\right)$, $T_i=\tau_i \wedge \mathcal{L}^{d}$.
	Then, classical arguments in measure theory show that the current~$T$
	has finite mass and
	\[
	\mathbb{M}_p(T) = \int_{\R^d}| \tau(x) |_{\mass,p}\, \d\mathcal{L}^{d}(x) . 
	\]
	In a similar spirit, given an $m$-dimensional rectifiable current~$T$ with coefficients in $(\Z^k,\| \cdot \|_{p})$, the mass of~$T$ can be written as
	\[
	\mathbb{M}_p(T) = \int_{\Sigma}\| \theta(x) \|_{p} \, \d\mathcal{H}^{m}(x).
	\]
	Moreover, $T$ can be expressed as $(T_1,\ldots,T_k)$, where each component $T_i$ is a $m$-dimensional classical rectifiable current with coefficients in $\Z$ (see~\cite{Ma}).
\end{remark}

We now recall Plateau's problem in the setting of currents with coefficients in the groups $(\R^{k}, \| \, \|_{p})$, $(\Z^{k}, \| \, \|_{p})$. Given a $(d-2)$-dimensional $\mathbb{Z}^k$-integral flat boundary~$S$ 
with compact support, we can define Plateau's problem for~$S$
in integral currents as
\begin{equation} \label{Plateau_int1}
    \begin{split}
	\mathbb{P}_{\Z^k,p}(S) := \inf\left\{\mathbb{M}_p(T) \colon T \textrm{ is a rectifiable } 
	\Z^k\textrm{-current of dimension } (d- 1), \ \partial T = S \right\} \! .
    \end{split}
\end{equation}
The Plateau problem for~$S$ in normal currents is
\begin{equation} \label{Plateau_normal1}
	\begin{split}
	\mathbb{P}_{\R^k,p}(S) := \inf\left\{\mathbb{M}_p(T) \colon T \textrm{ is a normal } 
	\R^k\textrm{-current of dimension } (d- 1), \ \partial T = S \right\} \! .
	\end{split}
\end{equation}

One of the advantages of the approach based on the theory of currents
is the possibility to define the (dual) notion of calibration,
which provides us with a tool to check whether
a given configuration is a minimizer.
We briefly recall the definition of calibration, which
will be useful in the next sections.
\begin{definition}\label{Calibration} 
		Consider a $\Z^k$-rectifiable current of dimension $m$, $T=\llbracket\Sigma, \tau, \theta\rrbracket$ in the ambient space $\R^{d}$. A smooth $(\R^{k})^{*}$-valued $m$-dimensional differential form $\omega$ in $\R^{d}$ is a calibration for $T$ if the following conditions hold:
		\begin{enumerate}
			\item[(i)]\label{clr1} $\langle \omega(x); \tau(x), \theta (x)\rangle= \| \theta(x) \|_{p}$ for $\mathcal{H}^{m}$-a.e~$x\in \Sigma$;
			\item[(ii)]\label{clr2} $\omega$ is closed, i.e, ${\rm d}\omega=0;$
			\item[(iii)]\label{clr3} for every $x\in \R^{d}$, for every simple unit vector $t \in  \Lambda_{m}(\R^{d})$ and for every $h\in \Z^{k}$, we have that
			$$\langle \omega(x); t, h \rangle \leq \|h \|_{p}.$$
		\end{enumerate}
\end{definition}
If a current $T$ admits a calibration, then it is a minimizer of the mass among all normal $\R^k$ currents its homology class. (See for instance \cite{Ma} for a proof of this claim). Therefore, the existence of a calibration for an integral~$\Z^k$-current~$T$ is a sufficient condition to prove that~$\mathbb{P}_{\Z^k,p}(\partial T) = \mathbb{P}_{\R^k,p}(\partial T)$.
The question of whether or not a minimizing current admits a calibration is a delicate issue (see for instance \cite[Section 1.4, Section 3.1.2]{Ma} and references therein, as well as \cite[Section 2.1 and Example 16]{BaLe}).

\subsubsection{Push-forward of currents}
We recall the notion of push-forward of currents. The push-forward
is first introduced for classical currents (see e.g.~\cite[Section 7.4.2]{Krantz2008}), then extended to currents
with coefficients in the group $(\R^k,\| \cdot \|_{p})$.
Let $f\colon\R^d\to \R^l$ be a smooth function, and let~$\omega$ be a smooth, compactly supported $m$-differential form in $\R^l$.
We define the pull-back of $\omega$, $f^{\#}(\omega)$, as a smooth differential form in $\R^d$, given by
\begin{equation}
 f^{\#}(\omega)(x)(v_1 \wedge \ldots \wedge v_m)
 :=\omega(f(x)) (\d f(v_1) \wedge \ldots \wedge \d f(v_m))
\end{equation}
for any $x\in \R^d$ and any simple $m$-vector $v=v_1 \wedge \ldots \wedge v_m$. 
Next, let $T$ be a classical normal current in $\R^d$. We assume again that~$f\colon\R^d\to \R^l$
is smooth and that $f\mid_{\spt(T)}$ is proper (i.e., for any compact set $A\subset \R^l$, $f^{-1}(A)\cap \spt(T)$ is compact).
Then, we define the push-forward of the current $T$, $f_{\#}(T)$,
to be the $m$-dimensional current in~$\R^l$ given by
\begin{equation}\label{pushfoward1}
f_{\#}(T)(\omega):=T(\psi f^{\#}(\omega))
\end{equation}
for any~$\omega \in C^{\infty}_{\mathrm{c}}(\R^{l}; \Lambda^{m}(\R^l))$,
where $\psi$ is any compactly supported, smooth function in $\R^d$
that is equal to $1$ in a neighborhood of $\spt(T)\cap \supp f^{\#}(\omega)$. It can be checked that the right-hand side of~\eqref{pushfoward1}
is independent of the choice of~$\psi$, and that
\begin{equation}
\partial f_{\#}(T)=f_{\#}(\partial T).
\end{equation}
This notion can be extended to the case~$f$ is a Lipschitz map (see \cite[Section 7.4.2, Lemma 7.4.3]{Krantz2008} for the details). Moreover, given a Lipschitz map~$f$ (with Lipschitz constant~$\Lambda$) such that the restriction of $f$ to the support of $T$ is proper, there holds
\begin{equation}
\mathbb{M}(f_{\#}(T)) \leq \Lambda \, \mathbb{M}(T)
\end{equation}
(see again \cite[Section 7.4.2, Lemma 7.4.3]{Krantz2008}).
The push-forward can be defined in case~$T = (T_1, \, \ldots, \, T_k)$
is a current with coefficients in the group 
$(\R^k, \|\cdot\|_{p})$, $(\Z^k, \|\cdot\|_{p})$
by working component-wise,
i.e.~$f_{\#}(T) := (f_{\#}(T_1), \, \ldots, \, f_{\#}(T_k))$.

\subsubsection{The Hodge-star operator and the codifferential}

The Hodge-star operator (on a Riemannian manifold of dimension~$d$)
is usually defined as a map between (classical)~$m$-covectors 
and~$(d-m)$-covectors. Here, instead, we find it convenient
to regard it as an operator between vectors and covectors.
These approaches are essentially equivalent;
the operator we consider here reduces to the classical
one, up to composition with the isomorphisms~$\#$, $\flat$.
However, for the convenience of the reader, 
we recall the definitions.

We define the Hodge-star operator as a map
\[
\star\colon \Lambda^m(\R^d)\to\Lambda_{d-m}(\R^d),
\]
in the following way: for any~$\beta\in\Lambda^m(\R^d)$,
$\star\beta$ is the unique element of~$\Lambda_{d-m}(\R^d)$ 
such that
\begin{align}\label{Hodge}
\langle \alpha, \star\beta \rangle 
:= \langle \beta \wedge \alpha,
\, e_1 \wedge \ldots \wedge e_d\rangle
\qquad \mbox{ for every }\alpha \in \Lambda^{d-m}(\R^{d}).
\end{align}
\begin{remark}
There is a variant definition of Hodge star operator which is defined by
\begin{align} \label{Hodge1}
\langle \alpha, \star\beta \rangle 
:= \langle \alpha \wedge \beta, 
\, e_1 \wedge \ldots \wedge e_d\rangle
\qquad \mbox{ for every }\alpha \in \Lambda^{d-m}(\R^{d})
\end{align}
for any~$\beta\in\Lambda^m(\R^d)$. 
The two definitions agree up to a sign. 
In this paper, we choose the definition in~\eqref{Hodge} for the compatibility in the sign involving the Hodge star operator, pre-jacobian, and distributional Jacobian later (see for instance~\eqref{compatibility}).
\end{remark}
An operator~$\Lambda_m(\R^d)\to\Lambda^{d-m}(\R^d)$,
still denoted~$\star$ for simplicity, is defined in a way
completely analogous to~\eqref{Hodge}.
The Hodge-star operator can be described explicitely
as follows: for any multi-index~$I$, there holds
\begin{align} \label{Hodge-explicit}
 \star\d x_I = \sigma(I, \, I^\prime) \,    e_{I^\prime}
 \qquad \textrm{and} \qquad
 \star   e_I = \sigma(I, \, I^\prime) \, \d x_{I^\prime}
\end{align}
where~$I^\prime$ is the set of indices that are not contained in~$I$
and~$\sigma(I^\prime, \, I)$ is the sign of the permutation~$(I^\prime, \, I)$.
From~\eqref{Hodge} and~\eqref{Hodge-explicit}, 
it follows that~$\star$ is an isometry and that
\begin{equation} \label{starstar}
 \star\star\beta = (-1)^{m(d-m)} \beta
\end{equation}
whenever~$\beta$ is an~$m$-vector or an~$m$-covector.
As a consequence, we have
\begin{equation} \label{Hodge_adj}
 \langle \alpha, \, \star\beta\rangle 
 = (-1)^{m(d-m)} \langle \star\alpha, \, \beta\rangle
\end{equation}
for any~$\alpha\in\Lambda^{d-m}(\R^d)$, $\beta\in\Lambda^m(\R^d)$
as well as any~$\alpha\in\Lambda_{d-m}(\R^d)$, $\beta\in\Lambda_m(\R^d)$.

The Hodge-star operator extends to an operator
between form-valued distributions and (classical) currents.
More precisely, let $\mathcal{D}^\prime(\R^d, \Lambda^m(\R^d))=(C^{\infty}_{\mathrm{c}}(\R^d, \Lambda_m(\R^d)))^\prime$ be the space of $m$-dimensional form-valued distributions
and let $\mathcal{D}^\prime(\R^d, \Lambda_{d-m}(\R^d))=(C^{\infty}_{\mathrm{c}}(\R^d, \Lambda^{d-m}(\R^d)))^\prime$ be the space of $(d-m)$-dimensional classical currents.
(Both $C^{\infty}_{\mathrm{c}}(\R^d, \Lambda_m(\R^d))$
and $C^{\infty}_{\mathrm{c}}(\R^d, \Lambda^m(\R^d))$ are equipped 
with the locally convex topology of test functions.)
In view of~\eqref{Hodge_adj}, it is natural to define
\[
\star\colon \mathcal{D}^\prime(\R^d, \Lambda^m(\R^d))\to \mathcal{D}^\prime(\R^d, \Lambda_{d-m}(\R^d)), \qquad
\star\colon \mathcal{D}^\prime(\R^d, \Lambda_m(\R^d))\to \mathcal{D}^\prime(\R^d, \Lambda^{d-m}(\R^d))
\]
by duality, in the following way: for any $\omega \in \mathcal{D}^\prime(\R^d, \Lambda^m(\R^d))$, the associated current $\star \omega$ is given by
\begin{equation} \label{Hodge-distr1}
\star \omega (\tau) :=(-1)^{m(d-m)} \omega(\star \tau)
\end{equation}
for any $\tau \in C^{\infty}_{\mathrm{c}}(\R^d, \Lambda^{d-m}(\R^d))$.
Similarly, given a current~$S\in \mathcal{D}^\prime(\R^d, \Lambda_m(\R^d))$,
we define the form-valued distribution~$\star S$ as
\begin{equation} \label{Hodge-distr2}
\star S (v) :=(-1)^{m(d-m)} S(\star v)
\end{equation}
for any $v\in C^{\infty}_{\mathrm{c}}(\R^d, \Lambda_{d-m}(\R^d))$.
The Hodge-star operator may be used to define the codifferential operator~$\d^*$:
\begin{align} \label{adjointoperatorofd}
 \d^{*} := (-1)^{d(m-1)+1} \star \d \star\colon C^{\infty}_{\mathrm{c}}(\R^d; \Lambda_{m}(\R^d))\to C^{\infty}_{\mathrm{c}}(\R^d; \Lambda_{m-1}(\R^d)).
\end{align}
While usually~$\d^*$ is defined as an operator between forms,
here we regard~$\d^*$ as an operator between vector fields.
This is consistent with our definition of the Hodge operator~$\star$, 
which maps vectors to covectors and vice-versa.
In spite of this unusual choice, the codifferential is still 
the formal adjoint of the exterior differential~$\d$ --- that is, for any $\alpha\in C^{\infty}_{\mathrm{c}}(\R^d; \Lambda^{m-1}(\R^d))$ 
and $\beta\in  C^{\infty}_{\mathrm{c}}(\R^d; \Lambda_{m}(\R^d))$, we have 
\begin{equation}
\int_{\R^d}\langle \d\alpha(x), \beta(x) \rangle \, \d x
= \int_{\R^d} \langle \alpha(x), \d^* \beta(x) \rangle \, \d x.
\end{equation}
We define the differential in the sense of distributions as follows:
given~$\omega\in\mathcal{D}^\prime(\R^d, \Lambda^m(\R^d))$,
we define~$\d\omega\in\mathcal{D}^\prime(\R^d, \Lambda^{m+1}(\R^d))$
by
\begin{equation} \label{d-distributions}
 \d\omega(v) := \omega(\d^* v)
\end{equation}
for any test vector-field~$v\in C^{\infty}_{\mathrm{c}}(\R^d, \Lambda_{m+1}(\R^d))$.
By comparing~\eqref{d-distributions} with the definition 
of the boundary operator~\eqref{boundary},
we deduce that the boundary and the codifferential
agree on smooth vector fields. More precisely,
if we identify a vector field~$v\in C^\infty_{\mathrm{c}}(\R^d, \, \Lambda_m(\R^d))$
with a current (defined by integration against~$v$), 
then~$\partial v$ is the current carried by~$\d^*v$.
(For instance, on~$1$-dimensional currents,
$\partial = -\Div$ in the sense of distributions.)
Moreover, the Hodge-star operator exchanges differentials and boundaries ---
that is, for any~$\omega \in \mathcal{D}^\prime(\R^d, \Lambda^m(\R^d))$,
and any~$T\in \mathcal{D}^\prime(\R^d, \Lambda_m(\R^d))$, there holds
\begin{align} 
 \star(\d\omega) &= (-1)^{m - 1} \, \partial(\star\omega), \label{partial-d} \\
 \d(\star T) &= - \, \star(\partial T). \label{partial-d-bis}
\end{align}
Equations~\eqref{partial-d} and~\eqref{partial-d-bis}
are a rather direct consequence
of the definitions and properties we have recalled in this section,
and we skip their proof.

For vector-valued currents, everything we said in
this section applies component-wise. 
For instance, if~$w\in\Lambda^m_{\R^k}(\R^d)$ and~$w$ 
can be written component-wise as~$w = (w_1, \, \ldots, \, w_k)$,
we define~$\star w := (\star w_1, \, \ldots, \, \star w_k)$.
In a similar way if~$v = (v_1, \, \ldots, \, v_k)\in\Lambda_{m,\R^k}(\R^d)$,
we set~$\star v := (\star v_1, \, \ldots, \, \star v_k)$.
The following property holds:
if~$\beta\in\Lambda^m(\R^d)$  (respectively, $\beta\in\Lambda_m(\R^d)$), 
$z\in(\R^k)^*$ (respectively, $z\in\R^k$)
and~$I\colon\R^k\to(\R^k)^*$ is the isomorphism
induced by the canonical basis in~$\R^k$, 
then~$\star(z\otimes\beta) = I^{-1}(z)\otimes\star\beta$
(respectively, $\star(z\otimes\beta) = I(z)\otimes\star\beta$).

\subsection{Convolution of currents}

\subsubsection{Convolution of currents with mollifiers}\label{convolutionwithmollifiers}

We recall the definition and basic properties of convolution
for currents. We will be interested only in
very special cases, i.e. convolution with
a sequence of mollifiers and with a vector field;
for a more general treatment, we refer e.g.~to~\cite{ABO1}.

In some of our technical results, we need
a radially symmetric, scale-invariant sequence of mollifiers.
Let~$\rho\in C^{\infty}_{\mathrm{c}}(\R^d)$
be a \emph{radial} function (i.e.~$\rho(x) = \bar{\rho}(|x|)$ for any~$x\in\R^d$, for some function~$\bar{\rho}\colon\R\to\R$)
such that~$0 \leq \rho \leq 1$, 
$\spt(\rho)\subseteq\bar{\mathbb{B}}^d_1(0)$
and $\int_{\R^d} \rho(x) \, \d x = 1$.
For~$\epsilon > 0$, we define
\begin{equation} \label{mollifier}
 \rho_\epsilon(x) := \epsilon^{-d} \rho\left(\frac{x}{\epsilon}\right)
 \qquad \textrm{for } x\in\R^d.
\end{equation}
Then, $\rho_{\epsilon}$ is a symmetric mollifier in $\R^{d}$, with
$\spt(\rho_{\epsilon})\subset \bar{\mathbb{B}}^d_{\epsilon}(0)$
and $\int_{\R^d} \rho_{\epsilon}(x) \, \d x=1$.
Let $\omega=\omega_1 e^*_1+\ldots+\omega_k e^*_k \in C^{\infty}_{\mathrm{c}}\left(\R^{d};\Lambda^{m}_{\R^{k}}(\R^{d})\right)$ be a test form. 
The convolution of $\omega$ and $\rho_{\epsilon}$ is defined component-wise as
\begin{equation}
(\omega * \rho_{\epsilon}) (x) 
:= \int_{\R^d} \omega (x-y)\rho_{\epsilon}(y) \, \d y
= \sum_{i=1}^{k} \left(\int_{\R^d} \omega_i (x-y)\rho_{\epsilon}(y) \, \d y\right)  e^*_i,
\end{equation}
for any~$x\in\R^d$.
Let $T=(T_1,\ldots,T_k)$ be an $\R^k$-normal current of dimension $m$ in $\R^d$. We define the convolution between $T$ and $\rho_{\epsilon}$ by
\begin{equation}
 (T*\rho_{\epsilon})(\omega):=T(\omega * \rho_{\epsilon})
\end{equation}
for any $\omega \in C^{\infty}_{\mathrm{c}}\left(\R^{d};\Lambda^{m}_{\R^{k}}(\R^{d})\right)$. As the convolution for test forms is defined component-wise,
we can think of~$T*\rho_\epsilon$ as 
being defined component-wise, too
(for the convolution of classical currents with a mollifier, 
see for instance \cite[Section~7.3, Definition $7.3.2$]{Krantz2008}). \label{Q2}
Then, $T*\rho_{\epsilon}=(T_{1,\epsilon},\ldots,T_{k,\epsilon})$, where each component $T_{i,\epsilon} = T_i*\rho_\epsilon$ is represented by a smooth $m$-vector-valued function in $\R^d$. 

\begin{lemma} \label{convolutionmollifier}
 For each $\R^k$-normal current~$T$, the following properties hold:
 \begin{itemize}
	\item[(i)] $T*\rho_{\epsilon} \overset{\ast}{\rightharpoonup} T$, that is, for each $\omega \in C^{\infty}_{\mathrm{c}}\left(\R^{d};\Lambda^{m}_{\R^{k}}(\R^{d})\right)$,
	we have $T*\rho_{\epsilon}(\omega) \to T(\omega)$;
	\item[(ii)] $\partial (T*\rho_{\epsilon})=\partial T * \rho_{\epsilon}$,
	\item[(iii)] for any $p\in [1, \, \infty]$ fixed,
	we have $\mathbb{M}_p(T*\rho_{\epsilon}) \to \mathbb{M}_p(T)$ 
	as~$\epsilon \to 0$.
\end{itemize}	
\end{lemma}
\begin{proof}
 Property~(i) is deduced from properties of classical currents 
 (see e.g.~\cite[Section~7.3, Lemma~7.3.3]{Krantz2008}) 
 by reasoning com\-po\-nent-\-wise. 
 Property~(ii) comes from the fact that
 \begin{equation}
  \d(\omega*\rho_{\epsilon})=\d\omega * \rho_{\epsilon}.
 \end{equation}
 To prove Property~(iii), we apply Jensen's inequality.
 Let $\omega \in C^{\infty}_{\mathrm{c}}\left(\R^{d};\Lambda^{m}_{\R^{k}}(\R^{d})\right)$. We have
 \begin{equation}\label{inequalitycomassnormconvolution}
  \begin{aligned}
   |\omega*\rho_{\epsilon}(x) |_{\comass, p}&=\left| \int_{\R^d}\omega(x-y)\, \rho_{\epsilon}(y) \, \d y \right|_{\comass, p}\\
   &\leq  \int_{\R^d}\left|\omega(x-y)\right|_{\comass, p} \, \rho_{\epsilon}(y) \, \d y \\
   &\leq \|\omega \|_{\comass, p} \int_{\R^d} \rho_{\epsilon}(y) \, \d y \\
   &= \|\omega \|_{\comass, p}.
  \end{aligned}
 \end{equation}
 Therefore,
 \begin{equation} \label{massconvolution0}
  \begin{aligned}
   \|\omega*\rho_{\epsilon} \|_{\comass, p} \leq \|\omega \|_{\comass, p}
  \end{aligned}
 \end{equation}
 for any~$p\in [1, \, \infty]$. In turn, \eqref{massconvolution0} implies that
 \begin{equation}\label{massconvolution1}
  \mathbb{M}_p(T)\geq \mathbb{M}_p(T*\rho_{\epsilon}).
 \end{equation}
 On the other hand, by the Property (i) we know that $T*\rho_{\epsilon} \overset{\ast}{\rightharpoonup} T$, and by lower semicontinuity of $\mathbb{M}_p$, we deduce 
 \begin{equation}\label{massconvolution2}
  \liminf_{\epsilon \to 0} \mathbb{M}_p(T*\rho_{\epsilon})\geq \mathbb{M}_p(T).
 \end{equation}
 From the inequalities~\eqref{massconvolution1} and~\eqref{massconvolution2}, we obtain~(iii).  
\end{proof}

In a similar way, we can define the convolution between $T$ and 
a bounded measure $\mu$ on $\R^d$, 
as
\begin{equation}
(T*\mu)(\omega):=T(\omega*\mu),
\end{equation}
where $(\omega*\mu)(x):=\int_{\R^d}\omega(x-y)\,\d\mu(y)$ for any test form~$\omega \in C^{\infty}_{\mathrm{c}}\left(\R^{d};\Lambda^{m}_{\R^{k}}(\R^{d})\right)$ and any~$x\in\R^d$.

\subsubsection{Convolution of currents with vector-valued maps and differential forms}
\label{convolutionwithvectorfunction}

Let $R=(R_1,\ldots,R_d)$ be a function in $L^{1}(\R^d, \R^d)$ with compact support such that $\Div R$ is a bounded measure on $\R^d$, and let~$T = (T_1, \, \ldots, \, T_k)$ be an $m$-dimensional normal current, with coefficients in the normed group $(\R^k, \| \cdot\|_{p})$. 
We aim to define the object
$T*R$, which will be an $\R^k$-normal current of dimension~$m+1$.
First, we briefly recall the contraction of covectors with vectors.
Let $\beta$ be an $h$-covector and $v$ be an $k$-vector with $h>k$.
We define the contraction $\beta \llcorner v$ to be 
an~$(h-k)$-covector, given by
\begin{equation}
\langle \beta \llcorner v, w \rangle := \langle \beta, (v \wedge w) \rangle
\end{equation}
for any $w \in \Lambda_{h-k}(\R^{d})$. 
Now, we define the convolution~$T*R$ by its action on a test form~$\omega = (\omega_1, \ldots, \omega_k) \in C^{\infty}_{\mathrm{c}}\left(\R^{d};\Lambda^{m+1}_{\R^{k}}(\R^{d})\right)$, as 
\begin{equation}\label{convolutionwithvector}
(T*R)(\omega) := \mathop{\sum_{i=1}^k} T_i(\bar{R}*\omega_i),
\end{equation}
where~$\bar{R}$ is the vector field given by~$\bar{R}(x):=R(-x)$, 
and the convolution $\bar{R}*\omega_i$ is defined by
\begin{equation}
 (\bar{R}*\omega_i)(x):=\int_{\R^d}\omega_{i}(x-y)\llcorner \bar{R} (y) \, \d y
\end{equation} 
for any~$x\in\R^d$. We observe that, for each~$i$, $\bar{R}*\omega_i$
is a smooth $m$-dimensional form with compact support, since $\omega_{i}$ is a smooth differential form with compact support and $R$ also has compact support.
Therefore, the right-hand side of~\eqref{convolutionwithvector}
is well-defined. 
From~\eqref{convolutionwithvector}, one has
\begin{equation} \label{massT*R}
\begin{aligned}
|T*R(\omega)|&\leq \mathop{\sum_{i=1}^k} \mathbb{M}(T_i)\| R*\omega_i \|_{L^{\infty}(\R^d)}\\
&\leq \mathop{\sum_{i=1}^k} \mathbb{M}_p(T)\| R * \omega_i \|_{L^{\infty}(\R^d)}\\
&\leq k\sqrt{k} \, \mathbb{M}_p(T)\|R\|_{L^{1}(\R^d)}
\end{aligned}
\end{equation}
for any~$\omega \in C^{\infty}_{\mathrm{c}}\left(\R^{d};\Lambda^{m+1}_{\R^{k}}(\R^{d})\right)$ such that $|\omega|_{\comass, p} \leq 1$.
Therefore,
\begin{equation} \label{massconvolution-vector}
\mathbb{M}_p(T*R)\leq k\sqrt{k}\,\mathbb{M}_p(T)\| R \|_{L^{1}(\R^d)}.
\end{equation}
Moreover, there holds
\begin{equation}\label{convolutionvector2}
\partial (T*R)=- \div R*T + (-1)^{m}R*\partial T
\end{equation}
in the sense of distributions. The equality~\eqref{convolutionvector2} 
follows by the analogous identity for classical currents
(see \cite[Section~2.8, Equation~(2.5)]{ABO1}), applied component-wise .

\begin{lemma} \label{lemma:smooth-convolution}
 Let~$T$ be a normal~$\R^k$-current of dimension~$m$.
 Let~$R\in L^1(\R^d, \, \R^d)$ be a vector field of compact support.
 Then, the current~$T*R$
 can be identified with an element of~$L^1(\R^d, \, \Lambda_{m+1,\R^k}(\R^d))$.
 Moreover, if~$R$ is smooth in~$\R^d\setminus\{0\}$,
 then (the vector field corresponding to)
 $T*R$ is smooth in~$\R^d\setminus\spt T$.
\end{lemma}
\begin{proof}
 As we have seen in~\eqref{massT*R}, the current~$T*R$
 has finite mass. By Riesz's representation theorem,
 it follows that~$T*R$ can be identified with a (regular, vector-valued) 
 Borel measure. We claim that~$T*R$ is absolutely 
 continuous with respect to the Lebesgue measure;
 then, the Radon-Nikodym theorem will imply that~$T*R$
 is carried by an integrable $\Lambda_{m+1, \R^k}(\R^d)$-field.
 Given a Borel set~$F\subseteq\R^d$, we will denote by
 $T*R\mres F$ the restriction of~$T*R$ to~$F$, which is well-defined 
 in the measure-theoretical sense, thanks to Riesz's 
 representation theorem.
 Let~$\varepsilon > 0$ be a small number.
 As~$R$ is integrable, there exists~$\delta > 0$ such that
 \begin{equation} \label{convsm1}
  \int_{E} \abs{R(x)} \, \d x \leq \varepsilon
  \qquad \textrm{for any measurable set } E\subseteq\R^d 
  \textrm{ such that } \abs{E}\leq \delta. 
 \end{equation}
 Let~$F\subseteq\R^d$ be a Borel set such that~$\abs{F} = 0$.
 As the Lebesgue measure is regular, there exists
 an open set~$U\supseteq F$ such that~$\abs{U} \leq \delta$.
 Let~$\omega\in C^\infty_{\mathrm{c}}(\R^d; \, \Lambda^{m+1}_{\R^k}(\R^d))$
 be a test form such that~$\spt\omega\subseteq U$
 and~$\norm{\omega}_{\comass, p} \leq 1$.
 Then, for any index~$i\in\{1, \, \ldots, \, k\}$
 and any~$x\in\R^d$, we have
 \[
  \abs{\bar{R}*\omega_i(x)}_{\comass}
  \leq \int_{\R^d} \abs{\omega_i(x - y)}_{\comass} \, \abs{R(-y)} \, \d y
  \leq \int_{-x - U} \abs{R(z)} \, \d z 
  \stackrel{\eqref{convsm1}}{\leq} \varepsilon
 \]
 (where the notation~$\abs{\cdot}_{\comass}$ stands for the classical 
 comass norm of a form). By reasoning as in~\eqref{massT*R},
 we deduce that 
 \begin{equation} \label{massT*R-added}
  \abs{T*R(\omega)}\leq k\sqrt{k}\,\varepsilon \, \mathbb{M}_p(T).
 \end{equation}
 Now, for any positive integer~$j$, let
 \[
  F_j := \left\{x\in F\colon\mathrm{dist}(x, \, \partial U)
   \geq \frac{1}{j} \right\}
 \]
 (where~$\mathrm{dist}(x, \, \partial U)$
 denotes the Euclidean distance of~$x$ from~$\partial U$)
 and let~$\varphi_j\in C^\infty_{\mathrm{c}}(U)$
 be a smooth cut-off function, such that~$\varphi_j = 1$ in~$F_j$
 and~$0 \leq \varphi_j \leq 1$ in~$U$. Given any form~$\widetilde{\omega}\in C^\infty_{\mathrm{c}}(\R^d; \, \Lambda^{m+1}_{\R^k}(\R^d))$
 such that~$\norm{\widetilde{\omega}}_{\comass, p} \leq 1$, the inequality~\eqref{massT*R-added} 
 (applied to~$\omega := \varphi_j\,\widetilde{\omega}$) implies
 \[
  \abs{(T*R\mres F_j)(\widetilde{\omega})}
  = \abs{(T*R\mres F_j)(\varphi_j\,\widetilde{\omega})}
  \leq k \sqrt{k} \, \varepsilon \, \mathbb{M}_p(T) .
 \]
 By taking the supremum over~$\widetilde{\omega}$, we obtain
 $\mathbb{M}_p(T*R\mres F_j)\leq k\sqrt{k} \, \varepsilon \, \mathbb{M}_p(T)$
 and, sending~$j$ to infinity,
 $\mathbb{M}_p(T*R\mres F)\leq k\sqrt{k} \, \varepsilon \, \mathbb{M}_p(T)$.
 As~$\varepsilon > 0$ is arbitrary,
 it follows that~$\mathbb{M}_p(T*R\mres F) = 0$
 and hence, $T*R$ is absolutely continuous
 with respect to the Lebesgue measure.
 Therefore, $T*R$ is carried by an integrable vector field,
 which we still denote by~$T*R$.
 
 Suppose now that~$R$ is smooth in~$\R^d\setminus\{0\}$.
 We claim that~$T*R$ is smooth in~$\R^d\setminus\spt T$.
 Let~$\delta > 0$, and let~$U_\delta := \{x\in\R^d\colon 
 \mathrm{dist}(x, \, \spt T) > \delta\}$.
 Let~$R_\delta$ be a vector field that is smooth
 everywhere in~$\R^d$ and coincides with~$R$ 
 in~$\R^d\setminus\mathbb{B}^d_\delta(0)$. 
 Finally, let~$\omega\in C^\infty_{\mathrm{c}}(\R^d, \, \Lambda^{m+1}_{\R^k}(\R^d))$ be a test form such that~$\spt\omega \subseteq U_\delta$.
 Then, for any index~$i\in\{1, \, \ldots, \, k\}$
 and any~$x\in\spt T$, we have
 \[
  \bar{R}*\omega_i(x)
  = \int_{x - U_\delta} \omega_i(x - y)\llcorner R(-y) \, \d y
  = \int_{x - U_\delta} \omega_i(x - y)\llcorner R_\delta(-y) \, \d y
  = \bar{R_\delta}*\omega_i(x),
 \]
 because~$x - U_\delta \subseteq\R^d\setminus\mathbb{B}^d_{\delta}(0)$.
 It follows that~$T*R (\omega) = T*R_\delta(\omega)$
 for any test form~$\omega$ with~$\spt\omega\subseteq U_\delta$
 and hence, $T * R = T * R_\delta$ in~$U_\delta$.
 As the convolution between a distribution and
 a smooth function (of compact support)
 is a smooth function, we deduce that~$T * R$
 is smooth in~$U_\delta$ for any~$\delta > 0$. 
 The lemma follows.
\end{proof}

\subsection{The distributional Jacobian}
We introduce the notion of distributional Jacobian for torus-valued maps. Let us recall first this notion in case of $\mathbb{S}^{1}$-valued maps. Let $u=(u^1,u^2) \in W^{1,1}_{\rm loc}(\R^{d}; \mathbb{S}^{1})$, we define the pre-jacobian $1$-form as
\begin{equation}\label{preJac}
j(u) :=u^1 \, \d u^2 - u^2 \, \d u^1
\end{equation}
and the distributional Jacobian as the $2$-form
\begin{equation} \label{Jac}
J(u) :=\frac{1}{2} \d j(u),
\end{equation}
where the differential is taken in the sense of distributions on $\R^d$. Moreover, we associate each $u \in W^{1,1}_{\rm loc}(\R^{d}; \mathbb{S}^{1})$ a $(d-1)$-dimensional (respectively, $(d-2)$-dimensional) classical current by $\star j(u)$ (respectively, $\star J(u)$), where $\star$ is the Hodge-star operator, as defined in~\eqref{Hodge}.

Now, for each torus-valued map $\u=(u_{1},\ldots,u_{k})\in W^{1,1}_{\rm loc}(\R^{d}; \mathbb{T}^{k})$ i.e. each $u_{i}\in W^{1,1}_{\rm loc}(\R^{d}; \mathbb{S}^{1})$, we define its pre-jacobian ${\bf j}\u$ and Jacobian ${\bf J}\u$ component-wise, i.e.
\begin{equation}
{\bf j}\u:=(j(u_{1}),\ldots,j(u_{k})),
\end{equation}
and
\begin{equation}
{\bf J}\u:=(J(u_{1}),\ldots,J(u_{k})).
\end{equation}
Moreover, we can associate to each $\u=(u_{1},\ldots,u_{k})\in W^{1,1}_{\rm loc}(\R^{d}; \mathbb{T}^k)$ a $(d-1)$-dimensional-$\R^k$ current, given by
\begin{equation}
\star\textbf{j}\u=(\star j(u_{1}),\ldots,\star j(u_{k})),
\end{equation}
and a $(d-2)$-dimensional-$\R^k$ current, given by
\begin{equation}
\star\mathbf{J}\u=(\star J(u_{1}),\ldots,\star J(u_{k})).
\end{equation}

\begin{remark}\label{compatibility}
 Let~$S$ be a~$\mathbb{Z}^k$-integral flat boundary.
 Let~$\mathcal{Q}_S$ be the set of 
 maps~$\u\in W^{1,1}_{\mathrm{loc}}(\R^d, \, \C^k)$
 that are constant in a neighbourhood of infinity
 and satisfy~$\star\mathbf{J}\u = \pi S$.
 For any $\u \in \mathcal{Q}_{S}$,
 we have
 \begin{align}
  \partial(\star\mathbf{j}\u) &= 2\pi \, S
 \end{align}
 thanks to~\eqref{partial-d}.
\end{remark}

\section{Energy-minimizing maps with prescribed singularities and mass minimization for normal currents}
\label{section3}

In this section, we focus on the proof of Theorem~\ref{A}.
Throughout this section, we always assume that~$S$ is a
$(d-2)$-dimensional integral flat boundary~$S$ 
with compact support and finite mass that satisfies the condition~\eqref{hp:H}. 
For the reader's convenience, we recall the functionals
we consider. Let~$\mathcal{Q}_S$ be defined as 
in Remark~\ref{compatibility}.
For~$\u=(u_{1},\ldots,u_{k})\in\mathcal{Q}_S$ and any~$p\in [1, \, \infty)$, 
we define
\begin{equation}
H_{p}(\u):=\int_{\R^{d}} (|\nabla u_1|^p+|\nabla u_2|^p+\ldots+|\nabla u_k|^p)^{\frac{1}{p}}\,\d x,
\end{equation}
while for~$p=\infty$
\begin{equation}
H_{\infty}(\u):=\int_{\R^{d}} \max\left(|\nabla u_1|, \, |\nabla u_2|, \, \ldots, \, |\nabla u_k|\right)\,\d x.
\end{equation}
We set
\begin{equation}\label{PrbH} 
\mathbb{H}_{p}(S):=\inf \left\{H_{p}(\u)\colon
\u=(u_{1},\ldots,u_{k})\in \mathcal{Q}_{S} \right\} \!.
\end{equation}
The second functional we consider
is defined in terms of the so-called nuclear norm of the gradient.
For~$p\in [1,\,\infty]$, the $p$-\textit{nuclear norm} of a 
linear map $A\colon\R^d \to \C^k$ 
is defined as follows:
\begin{equation}\label{normonthegradient}
\abs{A}_{\nucl,p} :=
\inf \left\{ \sum_{i=1}^l \|z_i\|_{p} \, |v_i|\colon A=\sum_{i=1}^l z_i \otimes v_{i}, \, z_i \in \C^k, \, v_{i} \mbox{ is a 1-covector in } \R^d \right\} \! .
\end{equation}
\begin{remark} \label{rk:nuclear}
 It may occasionally be convenient to
 consider~$A$ as a linear map~$\R^d\to\R^{2k}$, 
 by identifying~$\C^k$ with~$\R^{2k}$, in the usual way. However,
 in the right-hand side of~\eqref{normonthegradient},
 $\|\cdot\|_p$ denotes the~$\ell^p$-norm in~$\C^k$,
 \emph{not}~$\R^{2k}$. In other words, given a 
 vector~$\zeta = (\zeta_1, \, \ldots, \, \zeta_k)\in\C^k$
 and~$p\in [1, \infty)$, we define~$\norm{\zeta}_p
 := \left(\sum_{j=1}^k \abs{\zeta_j}^p\right)^{1/p}$, 
 where~$\abs{\zeta_j}$ is the Euclidean norm of the $j$-th 
 component~$\zeta_j\in\C\simeq\R^2$. Similarly, 
 we define~$\norm{\zeta}_\infty := \max_{1\leq j\leq k}\abs{\zeta_j}$.
\end{remark}
The terminology `nuclear norm'
refers to the notion of nuclear linear operators
between topological vector spaces,
which has been introduced by Grothendieck~\cite{Gron1955}
in much greater generality. Here, of course,
the setting is much simpler because we only need to
consider linear operators between the finite-dimensional
Banach spaces~$(\R^d, \, \abs{\cdot})$ and~$(\C^k, \, \norm{\cdot}_p)$.
For a detailed discussion on the nuclear norm of operators between
Banach spaces, see e.g.~\cite{Jameson}. 
 
\begin{remark} \label{rk:nuclear-info}
 The choice of the nuclear norm
 can be explained by the fact that the nuclear norm~$\abs{\cdot}_{\nucl,p}$
 is dual to the standard operator norm on the space of
 linear maps $(\C^k, \, \norm{\cdot}_p)\to(\R^d, \, \abs{\cdot})$
 (see e.g.~\cite[Proposition~1.11]{Jameson}).
 Heuristically, just as the operator norm `is analogous'
 to the comass norm on~$\Lambda^1_{\R^k}(\R^d)$,
 so does the nuclear norm `correspond' to the mass norm 
 on~$\Lambda_{1,\R^k}(\R^d)$.
 (For a precise statement, see 
 Lemma~\ref{equivalentmassnormandgradient} below.)
 In case~$p=2$, the nuclear norm reduces 
 to the so-called trace norm or $1$-Schatten norm:
 for any linear map~$A\colon \R^d\to\C^k$ we have
 \[
  \abs{A}_{\nucl, 2} = \mathrm{trace}\left(\,(A^*\,A)^{1/2}\right)
  = \sum_{i=1}^{s} \abs{\lambda_i}
 \]
 where~$A^*$ is the adjoint of~$A$, $s:= \min(d, \, 2k)$
 and~$\lambda_1, \, \ldots, \, \lambda_s$
 are the singular values of~$A$.
\end{remark}

We set
\begin{equation}
E_{p}(\u):=\int_{\R^{d}} |\nabla \u (x)|_{\nucl,p}\,\d x,
\end{equation}
where $\u=(u_{1},\ldots,u_{k})\in \mathcal{Q}_{S}$, 
and
\begin{equation}\label{PrbE}
\mathbb{E}_{p}(S):=\inf \left\{ E_{p}(\u)\colon
\u=(u_{1},\ldots,u_{k})\in \mathcal{Q}_{S} \right\} \! .
\end{equation}
\begin{remark}
 When~$k=1$, the problems~\eqref{PrbH} and \eqref{PrbE} reduce to the one studied in~\cite{BrezisCoronLieb, Brezis2021}.
 When $d=2$, our energy $E_{p}$ belongs to the class of energies investigated in~\cite{BaLe}, where the authors make a connection with branched optimal transport theory. Therefore, in the planar case we provide information
 on the relationship between energy minimizing torus-valued maps with prescribed topological points singularities and the branched optimal transport problem or Gilbert-Steiner problem between those points.
 This is in the spirit of the work by Brezis, Coron and Lieb~\cite{BrezisCoronLieb}, which provides the connection between 
 the energy of harmonic maps and the optimal transport Monge-Kantorovich problem.
\end{remark}
We start with some auxiliary results. 

\begin{lemma} \label{lemma:integer}
 Let~$R$ be a \emph{classical} integer-multiplicity
 flat boundary in~$\R^d$, of dimension $(d-2)$, with compact support.
 Let~$T$ be a classical current of dimension~$(d-1)$,
 of finite mass, such that~$\partial T = R$.
 Assume that~$T$ is represented by an integrable vector field
 (still denoted~$T$) and 
 that~$T\in C^\infty(\R^d\setminus\spt R, \, \Lambda_{d-1}(\R^d))$.
 Let~$\gamma\colon\mathbb{S}^1\to\R^d$ be a Lipschitz loop,
 such that $\gamma(\mathbb{S}^1)\subseteq\R^d\setminus\spt R$.
 Then, the integral
 \[
  \int_{\gamma(\mathbb{S}^1)} \star T 
 \]
 is an integer number.
\end{lemma}
\begin{proof}
 The Hodge dual~$\star T$ is a differential $1$-form
 and is smooth away from the support of~$R$, so the integral
 of~$\star T$ on~$\gamma(\mathbb{S}^1)$ is well-defined.
 In order to prove that the integral is an integer number,
 we can assume without loss of generality that~$R$ is a 
 polyhedral current. Indeed, let~$\epsilon > 0$ be a small number.
 Even if~$R$ is not polyhedral, by Federer's polyhedral
 approximation theorem~\cite[\S~4.2.20 and~4.2.22]{FeBook}
 there exist an integer-multiplicity polyhedral 
 current~$P_1$ and finite-mass, integer-multiplicity
 currents~$Q_1$, $Q_1^\prime$, all of them
 supported in an~$\epsilon$-neighbourhood
 of~$\spt R$, such that~$R - P_1 = Q_1 + \partial Q_1^\prime$.
 As~$\partial R = 0$, we have~$\partial Q_1 = - \partial P_1$
 and in particular, the boundary of~$Q_1$ is polyhedral.
 Therefore, by Federer and Fleming's deformation theorem~\cite[\S~4.2.9]{FeBook},
 we can write~$Q_1 = P_2 + \partial Q_2$, where~$P_2$, $Q_2$
 are integer-multiplicity currents of finite mass supported
 in an~$\epsilon$-neighbourhood of~$\spt Q_1$
 and~$P_2$ is polyhedral. Let~$\tilde{T} := T - Q_1^\prime - Q_2$.
 Then, $\partial \tilde{T} = R - \partial (Q_1^\prime + Q_2) = P_1 + P_2$
 is polyhedral and, since we have assumed that
 $\gamma(\mathbb{S}^1)$ is a compact subset of~$\R^d\setminus\spt R$,
 we can take~$\epsilon$ small enough so that~$\tilde{T} = T$
 in a neighbourhood of~$\gamma(\mathbb{S}^1)$. Therefore,
 up to replacing~$T$ with~$\tilde{T}$, there is no loss of generality
 in assuming that~$\partial T = R$ is polyhedral.
 Moreover, it is not restrictive to assume that~$\gamma(\mathbb{S}^1)$
 is the boundary of a $2$-dimensional polyhedron~$\Sigma$.
 Indeed, the form~$\star T$ is closed in the complement
 of~$\spt R$ (due to Equation~\eqref{partial-d-bis})
 and~$\gamma(\mathbb{S}^1)$ can be approximated
 by polyhedral boundaries, each of which is a finite sum
 of boundaries of individual polyhedra.
 Finally, up to a small perturbation, we can 
 take~$\Sigma$ transverse to each simplex in~$R$.
 Then, for each $(d - 2)$-dimensional simplex~$K$ in~$R$,
 the intersection~$K\cap \Sigma$ contains at most one point.
 
 Let us assign an orientation to each $(d-2)$-simplex~$K$
 in~$R$, by considering a unit $(d-2)$-vector~$\tau_K$ 
 that spans the plan of~$K$.
 Let~$\tau_\Sigma$ be a unit $2$-vector
 that spans the plan of~$\Sigma$. 
 We define the intersection number~$I(K, \, \Sigma)$
 as follows: $I(K, \, \Sigma) := 0$ if~$K\cap\Sigma$ is empty;
 $I(K, \, \Sigma) := 1$ if~$K\cap\Sigma$ is non-empty
 and~$\star(\tau_K\wedge\tau_\Sigma) > 0$;
 and~$I(K, \, \Sigma) := -1$ otherwise. 
 Let~$\theta(K)\in\mathbb{Z}$ be the multiplicity
 of~$R$ along the (oriented) simplex~$K$. We claim that
 \begin{equation} \label{integer1}
  \int_{\partial\Sigma} \star T =
  \sum_{K} \theta(K) \, I(K, \, \Sigma),
 \end{equation}
 where the sum is taken over all $(d-2)$-simplices~$K$ of~$R$.
 Indeed, let~$\epsilon > 0$ be a small parameter
 and let~$\rho_\epsilon$ be a mollifier, as in~\eqref{mollifier}.
 Stokes' theorem implies
 \begin{equation*}
  \begin{split}
   \int_{\partial\Sigma} \star (T*\rho_\epsilon) 
   &= \int_{\Sigma} \d(\star (T*\rho_\epsilon))
    \stackrel{\eqref{partial-d-bis}}{=}
   - \int_{\Sigma} \star \partial (T*\rho_\epsilon)
  \end{split}
 \end{equation*}
 and hence, by Lemma~\ref{convolutionmollifier},
 \begin{equation} \label{integer2}
  \begin{split}
   \int_{\partial\Sigma} \star (T*\rho_\epsilon) 
   &= - \int_{\Sigma} \star (R*\rho_\epsilon).
  \end{split}
 \end{equation}
 The right-hand side of~\eqref{integer2} can be 
 further evaluated by writing the convolution~$R*\rho_\epsilon$
 as an integral (which is possible, because~$R$
 is polyhedral). Then, \eqref{integer1} follows 
 by taking the limit as~$\epsilon\to 0$
 in both sides of~\eqref{integer2}, 
 and~\eqref{integer1} implies the conclusion of the lemma.
\end{proof}
\begin{lemma}\label{mainlemma}
Let $T=(T_1,\ldots,T_k)$ be an $\R^k$-normal current of dimension $(d-1)$ and having compact support, such that $\partial T=S$.
Let~$p\in [1, \, \infty]$ be fixed.
Then, the following statements hold.
\begin{itemize}
	\item[(i)] For any $\eta >0$, there exists an $\R^k$-normal current of dimension $(d-1)$, $\bar{T}=(\bar{T}_{1},\ldots,\bar{T}_{k})$, such that $\bar{T}_{i}\in  C^{\infty}\left(\R^{d}\setminus \spt S_i;\Lambda_{d-1}(\R^{d})\right)$ 
	for any index~$i$, $\partial \bar{T}=S$ and
	$\mathbb{M}_p(\bar{T})\leq \mathbb{M}_p(T)+\eta$.
	\item[(ii)] Moreover,
	there exists $\u\in \mathcal{Q}_{S}$ such that 
	\begin{equation} \label{mainlemma0}
	\frac{1}{2\pi}\star {\bf j}(\u) = \bar{T}.
	\end{equation} 
\end{itemize}
\end{lemma}
\begin{proof}
\mbox{}

\medskip
\noindent
\textbf{Proof of $(i)$.}
We prove this fact by using the properties of the convolution, introduced in Section \ref{convolutionwithmollifiers}, \ref{convolutionwithvectorfunction}.
Let~$\rho_{\epsilon}$ be a symmetric mollifier kernel, exactly as in Section~\ref{convolutionwithmollifiers}. We consider the convolution~$T*\rho_{\epsilon}$. Each component of~$T*\rho_{\epsilon}$
is a classical current carried by a smooth $(d-1)$-vector-valued function with compact support in $\R^d$. Moreover, for~$\epsilon$
small enough, the mass of~$T*\rho_{\epsilon}$
is arbitrarily close to the mass of~$T$
(see Property~(iii) in Section~\ref{convolutionwithmollifiers}).
However, we cannot choose~$T_\epsilon = T*\rho_{\epsilon}$,
because the boundary of~$T*\rho_{\epsilon}$
is not necessarily equal to~$S$ --- in fact, Property~(ii)
in Section~\ref{convolutionwithmollifiers} gives
\begin{equation} \label{mainlemma1}
 \partial(T*\rho_\epsilon) = S *\rho_\epsilon.
\end{equation}
Therefore, we need to modify~$T*\rho_\epsilon$ in a suitable way,
so as to obtain a new current whose boundary is equal to~$S$.

By Lemma~\ref{laplacian} in the appendix, 
there exists a vector-valued map 
$R_{\epsilon}\in L^1(\R^d; \R^d)$ that is smooth in~$\R^d\setminus\{0\}$,
satisfies
\begin{equation}
\Div R_{\epsilon}=\delta_{0}-\rho_{\epsilon}
\end{equation}
and $\|R_{\epsilon}\|_{L^{1}(\R^d)} \to 0$, $\spt(R_{\epsilon})\subset \bar{\mathbb{B}}_{\epsilon}(0)$. 
We consider the convolution~$S * R_\epsilon$,
defined in Section~\ref{convolutionwithvectorfunction}.
This is a~$(d-1)$-dimensional normal $\R^k$-current,
and the~$i$-th component of~$S* R_\epsilon$ is carried by a vector field
that is smooth in the complement of~$\spt S_i$, 
by Lemma \ref{lemma:smooth-convolution}.
Moreover, \eqref{convolutionvector2} implies
\begin{equation} \label{mainlemma2}
\partial (S*R_{\epsilon})
= (-1)^{d-2}  R_{\epsilon}*\partial S- S*\Div R_{\epsilon}
= \rho_{\epsilon}*S-S.  
\end{equation}
Let $T_{\epsilon} := T*\rho_{\epsilon}-S*R_{\epsilon}$.
By~\eqref{mainlemma1} and~\eqref{mainlemma2},
we have $\partial T_{\epsilon}=S$.
Moreover, each component $T_{\epsilon}^i$ of~$T_\epsilon$
is carried by a vector-field in $C^{\infty}_{\mathrm{c}}\left(\R^{d}\setminus\spt S_i;\Lambda^{d-1}(\R^{d})\right)$ and, thanks to~\eqref{massconvolution1}, 
\eqref{massconvolution-vector}, we have
\begin{equation}
\begin{aligned}
\mathbb{M}_p(T_{\epsilon}) 
&\leq \mathbb{M}_p(T*\rho_{\epsilon}) + \mathbb{M}_p(S*R_{\epsilon}) \\
&\leq \mathbb{M}_p(T) + \|R_{\epsilon}\|_{L^1(\R^d)} \, \mathbb{M}_p(S) \\
&\leq \mathbb{M}_p(T) + \mathrm{o}(1) 
\end{aligned}
\end{equation}
as~$\epsilon \to 0$. Therefore, taking~$\epsilon$ small enough
(depending on~$\eta$), the current~$\bar{T} := T_\epsilon$
has all the required properties.

\medskip
\noindent
\textbf{Proof of $(ii)$}.
We recall that~$G_i$ is the subset of $W^{1,1}_{\rm loc}(\R^{d}; \mathbb{S}^{1})$ consisting of the maps~$u_{i}$ 
that are constant outside an open ball
of radius $r_i = r_{i}(u_i)$ 
and are such that $\star J(u_i)=\pi S_{i}$.
We shall prove that, for each $i=1,\ldots, k$, there exists $u_i\in G_i$ such that
\begin{equation} \label{mainlemma-jacu}
\frac{1}{2\pi}\star j(u_i)=T_{\epsilon, i}.
\end{equation}
Then, choosing~$\u = (u_1, \, \ldots, ,\ u_k)$,
the lemma will follow.
Let fix a point $x_0$ in $\R^d \setminus \spt S_i$.
The set~$\R^d\setminus\spt S_i$ is open and, by assumption,
connected. Therefore, for any $x\in \R^d \setminus \spt S_i$ 
there exists a path $\gamma^i_{x}\colon [0, \, 1]$,
whose support is contained in $\R^d \setminus \spt S_i$, 
that connects $x$ and $x_0$. Define
\begin{equation} \label{construction_u-i}
\theta_{i}(x):=(-1)^{d-1}\int_{\gamma^i_x}\star T_{\epsilon,i},
\end{equation}
then let $u_{i}(x):=e^{2\pi i \theta_{i}(x)}$.
The quantity~$u_i(x)$ is well-defined
and independent of the choice of~$\gamma_x^i$.
Indeed, if $\gamma_{1}, \gamma_{2}$ are two paths 
connecting $x$ and $x_0$, Lemma~\ref{lemma:integer} implies 
\begin{equation}
\int_{\gamma_{1}}\star T_{\epsilon,i}
-\int_{\gamma_{2}}\star T_{\epsilon,i} \in \mathbb{Z}.
\end{equation}

We now prove that the map $u_i\colon\R^d\setminus\spt S_i\to\mathbb{S}^1$ 
is smooth in $\R^d \setminus \spt S_i$. Take an arbitrary small open ball centered at $y_0$ with radius $r$, $\mathbb{B}^d_r(y_0)\subset\R^d \setminus \spt S_i$. 
Thanks to \eqref{partial-d}, one has 
\[
 (-1)^{d-1}\,\d(\star T_{\epsilon,i})
 = \partial T_{\epsilon, i} = S_i = 0
 \qquad \textrm{in } \mathbb{B}^d_r(y_0).
\]
Therefore, there exists a smooth function $\phi_i\colon\mathbb{B}^d_r(y_0) \, \to \R$ such that $\d\phi_i=(-1)^{d-1}\star T_{\epsilon,i}$. 
Up to an additive constant, we can assume that~$\varphi_i(y_0) = 0$.
For any $x\in \mathbb{B}^d_r(y_0)$, let $\sigma_x$ be the path connecting $y_0$ to $x$ along the radius of the ball. By definition of~$u_i$, one has
\begin{equation}
u_i(x) = u_i(y_0) \, e^{2\pi i\bar{\phi}_i(x)},
\end{equation}
where 
\begin{equation}
 \bar{\phi}_i(x)
 := (-1)^{d-1}\int_{\sigma^i_{x}}\star T_{\epsilon,i}
  = \int_{\sigma^i_{x}} \d\phi_i = \phi_i(x).
\end{equation}
Therefore, $u_i(x) =  u_i(y_0) \, e^{2\pi i\phi_i(x)}$.
This proves that~$u_i$ is smooth in~$\mathbb{B}^d_r(y_0)$,
because~$\phi_i$ is. Moreover,
by explicit computation, we obtain
\begin{equation}
j(u_i)=2\pi \d\phi_i=2\pi_i(-1)^{d-1}\star T_{\epsilon,i}
\end{equation}
and~\eqref{mainlemma-jacu} follows.
Due to the special structure of the target manifold~$\mathbb{S}^1$, 
we have, for each $x\in \R^d$,
\begin{equation}\label{equalityoperatornormcomssnorm}
\left|\star j(u_i)(x) \right|
=\left| j(u_i)(x) \right| = \left| \nabla u_i(x) \right|
\end{equation}
(we recall that the notation~$|\cdot|$, without any subscript, 
indicates the Euclidean norm; see Remark~\ref{remarknotation}). From~\eqref{mainlemma-jacu} and~\eqref{equalityoperatornormcomssnorm}, combined with the fact that $T_{\epsilon, i}$ has finite mass, 
it follows that $u_i \in W^{1,1}(\R^d\setminus\spt(S_i),\mathbb{S}^1)$.
On the other hand, a general property of Sobolev spaces
is that~$W^{1,1}(\R^d\setminus F) = W^{1,1}(\R^d)$
for any closed set~$F\subseteq\R^d$ such that~$\mathcal{H}^{d-1}(F) = 0$
(see e.g.~\cite[Theorem~1.2.5 p.~16]{MazyaPoborchi}).
Therefore, taking Assumption~\eqref{hp:H} into account, 
we deduce that $u_i\in W^{1,1}(\R^d,\mathbb{S}^1)$.
Moreover, \eqref{mainlemma-jacu} and~\eqref{partial-d} imply
\begin{equation}
 \star J(u_i) = \frac{1}{2} \star(\d j(u_i))
 = \frac{1}{2} \partial (\star j(u_i))
 = \pi \, \partial T_{\epsilon,i} = \pi S_i.
\end{equation}
Finally, since $T_{\epsilon,i}$ has compact support, 
it follows by construction that
$u_i$ is constant outside some open ball
that contains~$T_{\epsilon,i}$
(see Equation~\ref{construction_u-i}).
This shows that~$u_i\in G_i$ and completes the proof.
\end{proof}

\begin{remark}\label{nonempty}
We deduce from Lemma~\ref{mainlemma}
that, if~$S$ has finite mass and 
the assumption~\eqref{hp:H} is satisfied, then
the space $\mathcal{Q}_{S}$ is non-empty.
\end{remark}

Before we state our next lemma, 
we introduce a notation:
for~$\omega\in\Lambda^1_{\R^k}(\R^d)$, we define
\begin{equation} \label{barN-norm}
 \bar{N}(\omega):=\inf \left\{ \sum_{i=1}^l \|z_i\|^*_{q} \, |\omega_i| \colon \omega=\sum_{i=1}^l z_i \otimes \omega_{i}, \, z_i \in (\R^k)^*, \, \omega_{i} \mbox{ is a $1$-covector in } \R^d \right\} \! ,
\end{equation}
where~$q\in [1, \infty]$ is the conjugate exponent of~$p$,
i.e. $1/p + 1/q = 1$. (We recall that~$\norm{\cdot}^*_q$
denotes the dual of the~$q$-norm on~$\R^k$;
therefore, $\norm{\cdot}^*_q$ is the~$p$-norm on~$(\R^k)^*$.)
The right-hand side of~\eqref{barN-norm} is formally analogous
to the characterization of the mass norm of a vector,
given by Equation~\eqref{massnorm-char}; however, the function~$\bar{N}$
is defined on forms, not vectors.
Finally, we recall that the mass norm~$\abs{\cdot}_{\mass,p}$
for vectors is defined by~\eqref{def_massnorm}.

\begin{lemma}\label{equivalentmassnormandgradient}
Let $\u=(u_1,\ldots,u_k)\in \mathcal{Q}_{S}$. 
Then, for a.e.~$x\in \R^d$, there holds
\begin{equation}
  \abs{\nabla \u(x)}_{\nucl,p}
  = \bar{N}(\mathbf{j}\u(x))
  = \abs{\star \mathbf{j}\u(x)}_{\mass,p}.
\end{equation}
\end{lemma}
\begin{proof}
Let~$x\in\R^d$ be a differentiability point for~$u$.
As $\mathbf{j}\u(x)$ is an $(\R^k)^*$-valued $1$-covector,
we have $\bar{N}(\mathbf{j}\u(x))
= \abs{\star \mathbf{j}\u(x)}_{\mass,p}$ 
(for more details, see Lemma~\ref{characterizationmassnorm} and Remark~\ref{characterizationmasscodimension1} in the appendix).
We shall prove that $\abs{\nabla \u(x)}_{\nucl,p}
= \bar{N}(\mathbf{j}\u(x))$.

\medskip
\noindent
\textbf{Step 1.} We first prove that 
\begin{equation} \label{eqmass1}
 \abs{\nabla \u(x)}_{\nucl,p}
  \geq \bar{N}(\mathbf{j}\u(x)).
\end{equation}
Consider a decomposition of~$\nabla\u(x)$ of the form
\begin{equation}\label{representationgradient}
\nabla \u(x)=\sum_{h=1}^\ell z_h \otimes v_{h},
\end{equation}
where $z_h = (z_j^1, \, \ldots, \, z_j^k) \in \C^k\simeq\R^{2k}$
and 
$v_h$ is a $1$-covector in $\R^d$ for any~$h=1,\ldots,\ell$. 
Let~$m\in\{1, \, \ldots, k\}$ be fixed.
We identify~$\nabla\u(x)$ with a real~$(2k)\times d$-matrix
and multiply on the left both sides of the equality~\eqref{representationgradient} by 
the (column) vector $(0,\ldots, iu_m(x), \ldots, 0)^{\mathsf{T}}\in \C^k\simeq\R^{2k}$,
in accordance with the rule of multiplication of matrices.
Recalling that Equation~\eqref{preJac}
can be written in the form~$j(u_m) = \langle i u_m, \, \d u_m \rangle$ 
(where~$\langle \cdot, \, \cdot \rangle$ stands
for the inner product between complex numbers), we obtain that
\begin{equation}\label{representationprejacobian}
j(u_m)(x)=\sum_{h=1}^\ell \langle iu_m(x), z^m_h \rangle \, v_{h}
\end{equation}
for each $m=1,\ldots,k$.
Therefore, it follows that
\begin{equation}
\mathbf{ju}(x)=\sum_{h=1}^\ell 
z_{\u, h}\otimes v_{h}
\qquad \textrm{where }
 z_{\u,h}:=\langle iu_1(x), z^1_j \rangle \, e^*_1
 + \ldots + \langle iu_k(x), z^k_j \rangle \, e^*_k\in(\R^k)^*.
\end{equation}
We observe that for any $h=1,\ldots, \ell$ and for any $m=1,\ldots,k$, 
by applying the Cauchy-Schwarz inequality we have
\begin{equation*}
\abs{\langle iu_m(x), z^m_h \rangle} 
\leq \abs{iu_m(x)} \, \abs{z^m_h} = \abs{z^m_h},
\end{equation*}
since $\abs{iu_m(x)} = 1$. Then,
\begin{equation}\label{representationprejacobian4}
\begin{aligned}
\sum_{h=1}^\ell \|z_h\|_{p} \, |v_h|
&\geq \sum_{h=1}^\ell\|z_{\textbf{u},h}\|^*_q \, |v_h|.
\end{aligned}
\end{equation}
From the equality~\eqref{representationprejacobian4},
combined with the definition of~$\bar{N}$
(Equation~\eqref{barN-norm}) and the definition of the nuclear norm
of a matrix~\eqref{normonthegradient}
(see also remark~\ref{rk:nuclear}), 
we conclude that~\eqref{eqmass1} holds.

\medskip
\noindent
\textbf{Step 2.} Now, we prove that
\begin{equation} \label{step2-equivalentnorms}
|\nabla \u(x) |_{\nucl,p}\leq \bar{N}(\mathbf{j}\u(x)).
\end{equation}
We first observe that, for any~$v
\in W^{1,1}_{\mathrm{loc}}(\R^d, \, \mathbb{S}^1)$, one has 
\begin{equation}\label{prejacobianandgradient}
 \nabla v(x) = i v(x)\otimes j(v)(x)
\end{equation}
for a.e.~$x\in\R^d$.
Indeed, by differentiating the constraint~$\abs{v}^2 = 1$ a.e.,
for any test vector~$e\in\R^d$
one obtains~$\langle\partial_e v(x), v(x)\rangle = 0$ for a.e.~$x\in\R^d$
and hence, $\partial_e v(x)$ must be parallel to~$iv(x)$ a.e.
Then, taking~\eqref{preJac} into account,
\eqref{prejacobianandgradient} follows.
Now suppose that 
\begin{equation}\label{representationofprejacobian}
\mathbf{ju}(x)=\sum_{h=1}^\ell z_h \otimes v_{h},
\end{equation}
where~$z_h=(z^1_h,\ldots,z^k_h)\in (\R^k)^*$ 
and~$v_h$ is a $1$-covector in $\R^d$, for~$h=1,\ldots,\ell$.
By applying~\eqref{prejacobianandgradient}, we deduce
\begin{equation*}
 \begin{split}
  \nabla\u(x) = \sum_{m=1}^k \nabla u_m(x) \, e^*_m
  = \sum_{m = 1}^k i u_m(x) \, e^*_m \otimes j(u_h)(x)
 \end{split}
\end{equation*}
and hence, thanks to~\eqref{representationofprejacobian},
\begin{equation} \label{representationnablau}
 \begin{split}
  \nabla\u(x) = \sum_{m = 1}^k \sum_{h = 1}^\ell
   i z_h^m u_m(x)  \, e^*_m \otimes v_h
   = \sum_{h = 1}^\ell \zeta_h\otimes v_h,
 \end{split}
\end{equation}
where~$\zeta_h := \sum_{m = 1}^k i z_h^m u_m(x) \, e^*_m \in\C^k$
(and, we recall, $z_h^m\in\R$ is the~$m$-th component of~$z_h$).
Since~$\abs{u_m(x)} = 1$, the Cauchy-Schwarz inequality
implies that $\norm{\zeta_h}_p \leq \norm{z_h}^*_q$ 
for any~$h\in\{1, \ldots, \, \ell\}$. Therefore, 
from~\eqref{normonthegradient}, \eqref{barN-norm}
and~\eqref{representationnablau}, we obtain~\eqref{step2-equivalentnorms}.
\end{proof}

\begin{proof}[Proof of Theorem \ref{A}]
We shall prove the equality~\eqref{equalityenergymimizingmapsandplateus} 
holds true for $1\leq p <\infty$ first, 
and deal with the case~$p = \infty$ later.

\medskip
\noindent
\textbf{Step 1.}
Let~$p\in [1, \, \infty)$ be fixed. In this step we shall prove that
\begin{equation}\label{equivalent1}
\mathbb{E}_{p}(S) = 2\pi \, \mathbb{P}_{\R^k,p}(S) .
\end{equation}
Let $\u \in \mathcal{Q}_{S}$.
From Lemma \ref{equivalentmassnormandgradient}, one has 
\begin{equation}\label{equivalent2}
\int_{\R^d} \abs{\nabla \u}_{\nucl,p} \, \d x
= \int_{\R^d} \abs{\star\mathbf{ju}}_{\mass,p} \, \d x.
\end{equation}
Moreover, $T:=\frac{1}{2\pi} \star \mathbf{ju}$ is an $\R^k$-normal current of dimension $(d-1)$ such that $\partial T=S$ 
(see Remark~\ref{compatibility}) and
\begin{equation}\label{equivalent3}
\mathbb{M}_{p}(T) = \frac{1}{2\pi}
\int_{\R^d} \abs{\star \mathbf{ju}}_{\mass,p} \, \d x
\end{equation}
(see Remark \ref{remarksoncurrents}). 
Therefore, in view of~\eqref{equivalent2} and \eqref{equivalent3},
we deduce that
\begin{equation}\label{equivalent4}
\int_{\R^d} \abs{\nabla \u}_{\nucl,p} \, \d x
= \int_{\R^d} \abs{\star \mathbf{ju}}_{\mass,p} \, \d x 
\geq 2\pi \, \mathbb{P}_{\R^k,p}(S).
\end{equation}
By taking the infimum over all $\u \in \mathcal{Q}_{S}$ in the inequality $\eqref{equivalent4}$, we obtain that
\begin{equation}\label{equivalent44}
\mathbb{E}_{p}(S) \geq 2\pi \, \mathbb{P}_{\R^k,p}(S).
\end{equation}
Now, let~$T$ be an $\R^k$-normal current such that $\partial T=S$ 
and $\mathbb{M}_p(T)\leq \infty$.
We are going to prove that for each $\eta >0$, 
there exists~$\u\in\mathcal{Q}_{S}$ such that
\begin{equation} \label{equivalent6}
2\pi \mathbb{M}_p(T)+\eta \geq E_{p}(\u).
\end{equation}
This will imply the opposite inequality 
\begin{equation}
2\pi \, \mathbb{P}_{\R^k,p}(S)\geq \mathbb{E}_{p}(S),
\end{equation}
and~\eqref{equivalent1} will follow.
Without loss of generality, we can assume that $T$ has compact support.
Indeed, since we have assumed that~$S$ has compact support,
we can consider a closed ball $\bar{\mathbb{B}}^d_r(0)$ with center at
the origin 
and radius $r>0$ such that $\spt S\subset \bar{\mathbb{B}}^d_r(0)$. 
Let $P_{r}\colon\R^d \to \bar{\mathbb{B}}^d_r(0)$ be the nearest-point projection onto~$\bar{\mathbb{B}}^d_r(0)$.
$P_r$ is a Lipschitz map with Lipschitz constant~$1$.
Then, the push-forward $P_{r\#}(T)=(P_{r\#}(T_1),\ldots,P_{r\#}(T_{k}))$
is an~$\R^{k}$-normal current of dimension $d-1$, 
with support contained in $\bar{\mathbb{B}}^d_r(0)$
and $\partial P_{r\#}(T) 
= P_{r\#}(\partial T) = S$. Moreover,
\begin{equation}\label{compactsupportnormalcurrent}
\begin{aligned}
\mathbb{M}_p(P_{r\#}(T)-T)
&\leq \mathop{\sum_{i=1}^k} \mathbb{M}_p(P_{r\#}(T_i)-T_i) 
\leq \mathop{\sum_{i=1}^k} \mathbb{M}_p(T_i \mres (\R^d \setminus \mathbb{B}_{r}(0)))
\end{aligned}
\end{equation}
(where~$\mres$ denotes the restriction).
The right hand side of~\eqref{compactsupportnormalcurrent} tends to zero,
and we obtain that~$P_{r\#}(T)$ converges to~$T$
with respect to the mass norm, as~$r\to\infty$.
Therefore, there is no loss of generality in 
assuming~$T$ is a current of compact support; but then,
\eqref{equivalent6} follows from Lemma~\ref{mainlemma}.
This completes the proof of~\eqref{equivalent1}.

\medskip
\noindent
\textbf{Step 2.}
Again, let~$p\in [1, \, \infty)$.
In this step we shall prove that
\begin{equation} \label{inequalityP2}
\mathbb{H}_{p}(S) \leq \mathbb{E}_{p}(S)\leq \mathbb{H}_{1}(S) 
\leq k^{1-\frac{1}{p}} \, \mathbb{H}_{p}(S).
\end{equation}
In fact, for any~$\u \in \mathcal{Q}_{S}$ there holds
\begin{equation} \label{equivalent11}
 H_p(\u) \leq E_{p}(\u)\leq H_{1}(\u)
 \leq k^{1-\frac{1}{p}} \, H_{p}(\u).
\end{equation}
Indeed, by using \eqref{comparison-massintro} (or see Lemma~\ref{comparenorms}),
Remark~\ref{remarksoncurrents} and
Lemma~\ref{equivalentmassnormandgradient} imply
\begin{equation}
 \int_{\R^d} \left(\sum_{i=1}^k 
  \abs{\star j(u_i)}^p\right)^{\frac{1}{p}} \d x 
 \leq E_p(\u) = \int_{\R^d} \abs{\star \mathbf{j}\u}_{\mass,p} \, \d x
 \leq \int_{\R^d} \sum_{i=1}^k 
  \abs{\star j(u_i)}_{\mass} \d x 
\end{equation}
where~$\abs{\star j(u_i)}$ is the 
Euclidean norm of the classical $(d-1)$-vector~$j(u_i)$.
However, all~$(d-1)$-vectors are simple,
therefore $\abs{\star j(u_i)}
= \abs{j(u_i)}$ a.e.~in~$\R^d$ 
and a standard computation shows 
that~$\abs{j(u_i)} = \abs{\nabla u_i}$ a.e.
This proves
\begin{equation}
 H_p(\u) \leq E_p(\u) \leq H_1(\u).
\end{equation}
Furthermore, the inequality
\begin{equation}
H_{1}(\u) \leq k^{1-\frac{1}{p}} H_p(\u)
\end{equation}
follows from Holder's inequality,
as for any $z\in \R^k$, one has
$\|z\|_{1}\leq k^{1-\frac{1}{p}}\|z\|_{p}$.

\medskip
\noindent
\textbf{Step 3.} 
To complete the proof, we shall check that the equality \eqref{equalityenergymimizingmapsandplateus} holds true in case~$p=\infty$.
To this aim, we will prove that $\mathbb{P}_{\R^k,p}(S)$ is a decreasing function of~$p$ 
and that
\begin{equation}\label{convergenceofPlateau}
\lim_{p\to \infty}\mathbb{P}_{\R^k,p}(S)=\mathbb{P}_{\R^k,\infty}(S).
\end{equation}
Similar properties hold for~$\mathbb{H}_{p}(S)$, $\mathbb{E}_{p}(S)$. 
We begin to prove~\eqref{convergenceofPlateau}. 
Let~$T$ be a normal $\R^k$-current of dimension~$d-1$.
According to Lemma~\ref{inequalityofmass} in the appendix,
for any $1\leq p_1 \leq p_2 \leq \infty$ there holds
\begin{equation}\label{inequalityformassp}
\mathbb{M}_{p_2}(T)\leq \mathbb{M}_{p_1}(T)\leq k^{\frac{p_2-p_1}{p_1p_2}}\mathbb{M}_{p_2}(T).
\end{equation}
The inequality \eqref{inequalityformassp} implies that
\begin{equation}\label{inequalityplateauproblem1}
\mathbb{P}_{\R^k,p_2}(S)\leq \mathbb{P}_{\R^k,p_1}(S)\leq k^{\frac{p_2-p_1}{p_1p_2}}\mathbb{P}_{\R^k,p_2}(S).
\end{equation}
Therefore, $\mathbb{P}_{\R^k,p}(S)$ is decreasing in~$p$
and bounded from below,
so its limit as~$p\to\infty$ exists and is unique. 
In particular, when $p_2=\infty$, $p_1=p\geq 1$, 
the inequality~\eqref{inequalityplateauproblem1} becomes
\begin{equation}\label{inequalityplateauproblem2}
\mathbb{P}_{\R^k,\infty}(S)\leq \mathbb{P}_{\R^k,p}(S)\leq k^{\frac{1}{p}}\, \mathbb{P}_{\R^k,\infty}(S).
\end{equation}
By taking the limit as~$p\to\infty$ in~\eqref{inequalityplateauproblem2}, 
we obtain that
\begin{equation}
\lim_{p\to \infty}\mathbb{P}_{\R^k,p}(S)=\mathbb{P}_{\R^k,\infty}(S).
\end{equation}
By making use of the inequalities
$\norm{z}_{p_2} \leq \norm{z}_{p_1} \leq k^{(p_2 - p_1)/(p_1p_2)} \norm{z}_{p_1}$ for any~$z\in\R^k$ and~$1 \leq p_1 \leq p_2 < \infty$,
in a similar way,
we also can prove that~$\mathbb{H}_{p}(S)$, $\mathbb{E}_{p}(S)$
are decreasing functions of~$p$ and that
\begin{align*}
 \lim_{p\to \infty}\mathbb{H}_{p}(S) &= \mathbb{H}_{\infty}(S), \\
 \lim_{p\to \infty}\mathbb{E}_{p}(S) &= \mathbb{E}_{\infty}(S).
\end{align*}
Finally, by taking the limit as~$p\to\infty$ in both sides 
of~\eqref{equivalent1} and~\eqref{inequalityP2}, 
we obtain that
\begin{equation}\label{equalityforcaseinfinity}
\mathbb{H}_{\infty}(S)\leq \mathbb{E}_{\infty}(S) = 2\pi \, \mathbb{P}_{\R^k,\infty}(S) \leq \mathbb{H}_{1}(S)\leq k\, \mathbb{H}_{\infty}(S),
\end{equation}
which completes the proof.
\end{proof}
\begin{remark} \label{rmthm1}
By reasoning as in Step~3,
we deduce that $\mathbb{P}_{\Z^k,p}(S)$ 
is a decreasing function of~$p$ and that
\begin{equation}
\lim_{p\to \infty}\mathbb{P}_{\Z^k,p}(S)=\mathbb{P}_{\Z^k,\infty}(S).
\end{equation}
Therefore, from Theorem~\ref{A}, we obtain the following inequality:
\begin{equation}\label{chaininequality}
\mathbb{H}_{p}(S)\leq \mathbb{E}_p(S)=2\pi \, \mathbb{P}_{\R^k,p}(S)
\leq 2\pi \, \mathbb{P}_{\Z^k,p}(S)\leq 2\pi \, \mathbb{P}_{\Z^k,1}(S)=\mathbb{H}_{1}(S)
\end{equation}
for any value of~$p\in [1, \, \infty]$. The equality
$2\pi \, \mathbb{P}_{\Z^k,1}(S)=\mathbb{H}_{1}(S)$ in~\eqref{chaininequality} can be established by using the~dipole construction and 
applying the coarea formula for each component-wise
(see~\cite[Remark 18]{BaLe});
moreover, in Section~\ref{section4} we provide an alternative proof
of the same equality (see Remark~\ref{RMequi}).
It is also natural to ask under which conditions Plateau's problems
$\mathbb{P}_{\Z^k,p}(S)$, $\mathbb{P}_{\R^k,p}(S)$ are equivalent 
to~$\mathbb{E}_p(S)$, $\mathbb{H}_p(S)$. 
This is the case, for instance, when~$p=1$, as all
the inequalities in~\eqref{chaininequality} reduce to
equalities when~$p=1$. Moreover, if a minimizer of the problem 
$\mathbb{P}_{\Z^k,p}(S)$ admits a calibration 
(see Definition~\ref{Calibration}), then 
$\mathbb{E}_p(S)$, $\mathbb{P}_{\Z^k,p}(S)$,
and $\mathbb{P}_{\R^k,p}(S)$ are equal to each other.
Indeed, the existence of a calibration implies that 
\begin{equation}\label{sufficentcondition}
\mathbb{P}_{\Z^k,p}(S)=\mathbb{P}_{\R^k,p}(S).
\end{equation}
From~\eqref{sufficentcondition}, with the aid of 
Theorem~\ref{A}, we obtain
\begin{equation}\label{sufficentcondition2}
\mathbb{P}_{\Z^k,p}(S)=\mathbb{P}_{\R^k,p}(S)=\mathbb{E}_p(S).
\end{equation}
\end{remark}

\begin{remark} \label{rk:noncalibration}
In case~$d=2$, the analysis contained in this section 
is consistent with~\cite{BaLe}, although we allow for 
slightly more general boundary data~$S$ than in~\cite{BaLe}
and, most importantly, we do not assume the existence of
a calibration for minimizers of~$\mathbb{P}_{\Z^k,p}(S)$.
However, we do not know if our results extend to maps 
that take their values in products
of higher-dimensional spheres, 
$\mathbb{S}^{m}\times\ldots\times \mathbb{S}^m$
with~$m\geq 2$.
Furthermore, from Theorem~\ref{A} we deduce that,
in the absence of a calibration, the conclusion of~\cite[Theorem~1]{BaLe}
(see Equation~\eqref{harmonic} in the introduction) may fail.
Let us consider the following example,
given in~\cite[Example $4.2$]{Bonafini2018}. 
Let~$d=2$, $k=4$, $p=\infty$, and let~$S$
be the $\Z^4$-current carried by the vertices of a regular
pentagon $P_1P_2P_3P_4P_5$, where the vertex~$P_i$
has multiplicity~$e_i = (0, \ldots, 1, \ldots 0)$ for~$i=1, \ldots, 4$
(the~$1$ is in the~$i$-th position) and~$P_5$
has multiplicity~$(-1, -1, -1, -1)$.
As shown in~\cite{Bonafini2018}, we have
\begin{equation} \label{strictineqPlateau}
2\pi \, \mathbb{P}_{\Z^k,\infty}(S)>2\pi \, \mathbb{P}_{\R^k,\infty}(S).
\end{equation}
This inequality, combined with Theorem \ref{A}, implies that
\begin{equation}
2\pi \, \mathbb{P}_{\Z^k,\infty}(S) > \mathbb{E}_{\infty}(S).
\end{equation}
This provides a counterexample for~\eqref{harmonic},
thus proving that the existence of a calibration
is an assumption that cannot be removed from~\cite[Theorem~1]{BaLe}.
Whether or not the same phenomenon happens also
for other values of~$p\in (1, \, \infty)$ is still an open question.
\end{remark}

As Equation~\eqref{strictineqPlateau} shows, the inequality
$\mathbb{P}_{\R^k,p}(S) \leq \mathbb{P}_{\Z^k,p}(S)$
may be strict, in general. Nevertheless,
in the next section we will derive a bound
for~$\mathbb{P}_{\Z^k,p}(S)$ in terms of~$\mathbb{P}_{\R^k,p}(S)$.
To this end, we will first write a characterization
for~$\mathbb{P}_{\Z^k,p}(S)$ in terms of torus-valued maps.

\section{BV-liftings of torus-valued maps and mass minimization for integral currents}
\label{section4}

In this section, we address the lifting problem for torus-valued map of bounded variation. As we will see later in the section, lifting results will allow us to obtain a bound for~$\mathbb{P}_{\Z^k,p}(S)$ in terms of~$\mathbb{P}_{\R^k,p}(S)$. 

In general terms, the lifting problem can be formulated as follows.
Let~$\mathcal{N}$ be a smooth, compact, connected Riemannian manifold without boundary and let $p\colon\mathcal{C}\to\mathcal{N}$ be the smooth Riemannian universal covering of~$\mathcal{N}$. Let~$\Omega$ be a smooth, bounded domain in $\R^d$ and~$\u\colon\Omega\to\mathcal{N}$ a measurable map.
A measurable map $\ttheta\colon\Omega\to\mathcal{C}$ is called a \textit{lifting} of~$\u$ if $p\circ\ttheta=\u$ a.e.~in~$\Omega$.  
We consider the following question:
given a regular map $\u\colon\Omega \to\mathcal{N}$, is there a lifting $\ttheta\colon\Omega\to\mathcal{C}$ of~$\u$ which has the same regularity of~$\u$? The answer depends on what kind of regularity we assume for~$\u$. For instance, whenever~$\Omega$ is a simple connected domain, a classical topological result guarantees that any~$C^k$-map
$\u\colon\Omega\to\mathcal{N}$ (with $k=0,\ldots, \infty$) has a~$C^k$ lifting.
The same result holds true for the other functional spaces, for instance 
the Sobolev space~$W^{1,p}$ with~$p\geq 2$ (see e.g.~\cite{BethuelZheng, BallZarnescu, Mucci}). However, the conclusion may fail for other functional spaces --- for instance, the Sobolev space~$W^{1,p}$ with $1\leq p< 2$.
(For example, when the domain $\Omega=\mathbb{B}^2_{1}(0)$ is the unit ball in~$\R^2$ and~$\mathcal{N}=\mathbb{S}^{1}$, the map~$u(x):=\frac{x}{|x|}$
belongs to~$W^{1,p}(\Omega, \, \mathbb{S}^1)$ but has no lifting in $W^{1,p}(\Omega, \R)$.) Therefore, we are naturally led to relax the regularity requirements on the liftings and look for a lifting in a larger space.
It was initially proved in~\cite{GiaquintaModicaSoucek} that
for any map~$u\in \BV(\Omega; \mathbb{S}^{1})$, there always exists a function~$\theta \in \BV(\Omega;\R)$ of bounded variation such that $ u=e^{i\theta}$ a.e. (here $\Omega$ does not need to be simply connected).
A similar lifting result can be established when the target 
is a general closed manifold~\cite{GiacomoOrlandi2}.

There are many contributions to the lifting problem (see for instance \cite{Ignat, Bourgain, Merlet, Bethuel, Bedford, Igna2t, MironescuVan} and the references therein).
In this paper, we consider as target space the 
manifold~$\mathcal{N}=\mathbb{T}^k=\mathbb{S}^1\times \ldots \times \mathbb{S}^1\subseteq\C^k$, with its covering space $\mathcal{C}=\R^k$ and the lifting map given by
\begin{equation}
\begin{aligned}
p \colon\R^k &\longrightarrow \mathbb{T}^k=\mathbb{S}^1\times \ldots \times \mathbb{S}^1 \\
(\theta_1,\ldots,\theta_k)&\longmapsto 
(e^{i\theta_1},\ldots,e^{i\theta_k}).
\end{aligned}
\end{equation}
We will also use the following notation for the covering map:
given~$\ttheta = (\theta_1, \, \ldots, \, \theta_k)$, we define
\begin{equation} \label{vector_e}
\textbf{e}^{\textbf{i}\ttheta}:= p(\ttheta) =
(e^{i\theta_1},\ldots,e^{i\theta_k})\in\mathbb{T}^k.
\end{equation}
Since we are interested in $\BV$-liftings (i.e., liftings that belong to the space of \textit{functions of Bounded Variation}), let us briefly recall the notion of $\BV$-space (see e.g.~\cite{Ambrosio} for more details). 
A function~$\ttheta\in L^{1}_{\mathrm{loc}}(\R^d,\R^k)$ is a function of locally bounded variation, $\ttheta\in \BV_{\mathrm{loc}}(\R^d, \R^k)$, if its distributional derivative $\D \ttheta$ is a vector-valued Radon measure 
on the $\sigma$-algebra of Borel sets of $\R^d$ and the total variation of~$\abs{\D\ttheta}(\Omega)$ is finite for any \emph{bounded}, open 
set~$\Omega\subseteq\R^d$. 
The derivative $\D \ttheta$ can then be decomposed as:
\begin{equation}
\D \ttheta=\D^{a}\ttheta+\D^{j}\ttheta+\D^{c}\ttheta,
\end{equation}
where 
\begin{itemize}
	\item $\D^{a}\ttheta$ is the absolutely continuous part of $\D \ttheta$ with respect to the Lebesgue measure,
	which can be written as~$\D^{a}\ttheta=\nabla \ttheta \, \mathcal{L}^{d}$, where $\nabla\ttheta$ is the approximate differential;
	\item $\D^{j}\ttheta$ is the jump part, which can be written as
	$\D^{j}\ttheta=(\ttheta^{+}-\ttheta^{-})\otimes n \, \mathcal{H}^{d-1}\llcorner \mathrm{S}(\ttheta)$, where $\mathrm{S}(\ttheta)$ is the set of approximate discontinuities (or jump set), oriented by the unit normal vector field $n$, and $\ttheta^{+}$, $\ttheta^{-}$ are the one-sided approximate limits of~$\ttheta$ on either side of~$\mathrm{S}(\ttheta)$;
	\item $\D^{c}\ttheta$ is the Cantor part of $\D \ttheta$,
	which is mutually singular to both~$\D^{a}\ttheta$ and~$\D^{j}\ttheta$.
\end{itemize}
The total variation of~$\ttheta\in\BV_{\mathrm{loc}}(\R^d, \R^k)$
on a bounded open set~$\Omega\subseteq\R^d$ can be represented as
\begin{equation}
|\D \ttheta|(\Omega)=\int_{\Omega}|\nabla \ttheta| \, \d x
+ \int_{\mathrm{S}(\ttheta)\cap \Omega}|\ttheta^{+}-\ttheta^{-}| \, \d \mathcal{H}^{d-1}
+ |\D^{c}\ttheta|(\Omega).
\end{equation}
We say that $\ttheta$ is a locally $\SBV$-function,
and we write $\ttheta\in\SBV_{\mathrm{loc}}(\R^d,\R^k)$,
if $\ttheta \in \BV_{\mathrm{loc}}(\R^d,\R^k)$ and $\D^{c}\ttheta=0$.
For any $\ttheta=(\theta_1,\ldots,\theta_k)\in \SBV_{\mathrm{loc}}(\R^d,\R^k)$, and any $p\in [1, \infty]$, we define
\begin{equation}
|\ttheta|_{\SBV,p}:=\int_{\R^d}\|\nabla \ttheta\|_p \, \d x
+ \int_{\mathrm{S} (\ttheta)}\|\ttheta^{+}-\ttheta^{-}\|_p \, \d \mathcal{H}^{d-1},
\end{equation}
where
\[
 \int_{\R^d}\|\nabla \ttheta\|_p\, \d x:=\int_{\R^d}\|(|\nabla \theta_1|,\ldots,|\nabla \theta_k|)\|_p \, \d x =\int_{\R^d}\left(|\nabla \theta_1|^p+\ldots+|\nabla \theta_k|^p\right)^{\frac{1}{p}} \, \d x,
\]
if~$p$ is finite; the case~$p=\infty$ is defined in a similar way.
In general, the value of the quantity $|\ttheta|_{\SBV,p}$ may be infinite.

Throughout this section, we consider again a $(d-1)$-dimensional 
integer-multiplicity flat boundary~$S$ with compact support
and finite mass.
Unless otherwise stated, we assume that~$S$ satisfies
the conditions~\eqref{hp:H} (however, the proof of 
Theorem~\ref{E} will not rely on the assumption~\eqref{hp:H}).
We consider a map~$\u = (u_1, \ldots, u_k)$ that belongs to~$\mathcal{Q}_S$
--- that is, $u\in W^{1,1}_{\mathrm{loc}}(\R^d, \, \C^k)$ is constant 
in a neighbourhood of infinity and satisfies $\star\mathbf{J}(\u) = \pi S$.
A map~$\ttheta=(\theta_1,\ldots,\theta_k)$ 
is said to be a $\BV$-lifting of~$\u$ if $\ttheta$ 
belongs to~$\BV_{\mathrm{loc}}(\R^d; \R^k)$ and $\bf{e}^{\bf{i}\ttheta}=\u$,
where~$\bf{e}^{\bf{i}\ttheta}$ is defined as in~\eqref{vector_e}.
Earlier results on the lifting problem for circle-valued
maps~\cite{GiaquintaModicaSoucek, Ignat} imply that any 
map $\u\in\mathcal{Q}_{S}$ admits a lifting
$\ttheta \in \BV_{\mathrm{loc}}(\R^d; \R^k)$ that satisfies 
a suitable estimate. 
\begin{theorem}\label{B}
	Let $\u=(u_1,\ldots,u_k)$ be in $\mathcal{Q}_{S}$. Then,  there exists a function $\ttheta=(\theta_1,\ldots,\theta_k) \in \SBV_{\mathrm{loc}}(\R^d; \R^k)$ such that
	\begin{equation} \label{Boundlifting-1}
	\mathbf{e}^{\mathbf{i}\ttheta}=\u
	\end{equation}
	and
	\begin{equation}\label{Boundlifting0}
	|\ttheta|_{\SBV,p} \leq 2k^{1-\frac{1}{p}}H_{p}(\u).
	\end{equation}
\end{theorem}

\begin{remark} \label{rk:liftingfirst}
The existence of a BV-lifting for torus-valued maps,
and more generally for maps with values in a closed manifold,
has been proved in~\cite[Theorem~1]{GiacomoOrlandi2}.
However, the results of~\cite{GiacomoOrlandi2} do not provide
an explicit estimate for the constant in~\eqref{Boundlifting0}.
Instead, we will apply a result by D\'avila and Ignat~\cite{Ignat},
who proved the existence of BV-liftings for circle-valued
maps and provided the optimal value of the constant in~\eqref{Boundlifting0}
when~$k=1$. 
We do not know whether the factor~$2k^{1-\frac{1}{p}}$ is optimal
for~$k > 1$ (see Theorem \ref{D}).
\end{remark}

\begin{proof}[Proof of Theorem~\ref{B}]
Applying the results of~\cite{Ignat}, for each~$j=1,\ldots,k$
there exists a function~$\theta_j\in \SBV_{\mathrm{loc}}(\R^d,\R)$ 
such that~$\theta_j$ is a lifting of~$u_j$, i.e.
\begin{equation}
 e^{i\theta_{j}}=u_j,
\end{equation}
and 
\begin{equation} \label{liftinginequality1}
|\theta_{j}|_{\BV} =\int_{\R^d} |\nabla \theta_j| \, \d x
 + \int_{\mathrm{S}(\theta_j)} |\theta^{+}_j-\theta^{-}_j| \, \d \mathcal{H}^{d-1}\leq 2\int_{\R^{d}}|\nabla u_j| \, \d x.
\end{equation}
Let $\ttheta=(\theta_{1},\ldots,\theta_{k})$. 
By taking the sum over~$j$ in both sides of~\eqref{liftinginequality1},
we obtain
\begin{equation} \label{liftinginequality2}
|\ttheta|_{\SBV,1} \leq 2 H_{1}(\u).
\end{equation}
Taking into account that, for any~$z\in\R^k$, 
the function $p\in [1, \, \infty]\mapsto \norm{z}_p$
is decreasing (see Lemma~\ref{decreasing}), we obtain
\begin{equation} 
|\ttheta|_{\SBV,p} \leq |\ttheta|_{\SBV,1}
\leq 2 H_{1}(\u) \leq 2 k^{1 - 1/p} \, H_{p}(\u),
\end{equation}
we have made use of~\eqref{liftinginequality2}
and applied the Holder's inequality.
\end{proof}

\begin{remark}\label{remarklifting1}
One has $|\nabla \theta_{j}(x)|=|\nabla u_{j}(x)|$ at points where $\theta_{j}$ is approximately differentiable. Indeed, let $x$ be a point of approximate differentiability for~$\theta$. By differentiating the identity $e^{i\theta_j}=u_j$,
we obtain for any $\ell=1,\ldots,d$
\begin{equation}
\frac{\partial\theta_j(x)}{\partial x_\ell} \, e^{i\theta_{j}(x)}
=\frac{\partial u_{j}(x)}{\partial x_\ell}.
\end{equation}
Thus,
\begin{equation}
\left|\frac{\partial\theta_j(x)}{\partial x_\ell}\right|
=\left|\frac{\partial u_j(x)}{\partial x_\ell}\right|,
\end{equation}
As a consequence, we deduce
\begin{equation}
|\nabla \theta_j(x)|=|\nabla u_j(x)|
\end{equation}
for a.e. $x\in \R^d$
and hence, 
\begin{equation}
\int_{\R^d}\|\nabla \ttheta\|_p \, \d x
= H_{p}(\u)
\end{equation}
for any~$p\in [1, \, \infty]$.
\end{remark}

We now address the relationship between the lifting problem
and mass minimization among integral currents,
and prove Theorem~\ref{C}.
For convenience, we recall some notation.
Given~$S$ as above, let us choose a 
map~$\u=(u_{1},\ldots,u_{k})\in \mathcal{Q}_{S}$.
This is possible because, under our assumption~\eqref{hp:H} for~$S$,
the set~$\mathcal{Q}_{S}$ is non-empty (see Remark~\ref{nonempty}).
We consider the following problem, for~$p\in [1, \, \infty]$:
\begin{equation} \tag{L} 
\mathbb{L}_p(S, \, \u) := \inf\left\{ \int_{\mathrm{S}(\ttheta)} 
\|\ttheta^+(x) - \ttheta^-(x)\|_{p} \,\d \mathcal{H}^{d-1}(x) \colon
\ttheta \textrm{ is a BV-lifting of } \u \right\}.
\end{equation}
First, we show that the infimum at the right-hand 
side of~\eqref{min_lifting} 
is independent of the choice of $\u$.

\begin{lemma} \label{lemma:indip_S}
 If~$\u\in\mathcal{Q}_{S}$, $\mathbf{v}\in\mathcal{Q}_{S}$
 (in particular, $\star \mathbf{J}(\u) = \star \mathbf{J}(\mathbf{v}) =\pi S$),
 then $\mathbb{L}_p(S,\u)=\mathbb{L}_p(S, \mathbf{v})$ for any~$p\in [1, \, \infty]$.
\end{lemma}
\begin{proof}
 Since $\, \star \mathbf{J}(\u) = \star \mathbf{J}(\mathbf{v}) =\pi S$,
 by an explicit computation we obtain
 \[
  J(u_j\overline{v_j})=J(u_j)-J(v_j)=0
 \]
 for any index~$j=1,\ldots, k$. Therefore,
 there exists $w_{j}\in W^{1,1}(\R^d;\R)$ such that 
 \begin{equation}
  u_{j}\overline{v_j}=e^{iw_{j}}
 \end{equation}
 (see for instance \cite[Lemma $1.8$]{Brezis2021}).
 On the other hand, we have $v_{j}\overline{v_j}=|v_{j}|^2=1$, which implies
 \begin{equation}
  u_{j}=e^{iw_{j}}v_{j}.
 \end{equation}
 Let $\ttheta=(\theta_{1},\ldots,\theta_{k})$ be a $\BV$-lifting of~$\u$.
 Then, $\bar{\ttheta}=(\theta_{1}-w_1,\ldots,\theta_{k}-w_k)$
 is a $\BV$-lifting of~$\mathbf{v}$ with the same jump as~$\ttheta$,
 since each~$w_j$ is a Sobolev function.
 Conversely, for any $\BV$-lifting of~$\mathbf{v}$ there 
 exists a $\BV$-lifting of~$\u$ with the same jump.
 Therefore, $\mathbb{L}_p(S,\u) = \mathbb{L}_p(S, \mathbf{v})$, as claimed.
\end{proof}

In view of Lemma~\ref{lemma:indip_S}, from now on we
write~$\mathbb{L}_p(S)$ instead of~$\mathbb{L}_p(S,\u)$.

\begin{remark} \label{rk:existence-minim-L}
 The infimum at the left-hand side of~\eqref{min_lifting}
 is always attained. Indeed, if~$(\ttheta_j)_{j\in\N}$
 is a minimizing sequence for~$\mathbb{L}_p(S)$, then
 the chain rule implies that~$\ttheta_j\in\SBV_{\mathrm{loc}}(\R^d, \, \R^k)$
 (if~$\ttheta_j$ had a nonvanishing Cantor part, then~$\u= \mathbf{e}^{\mathbf{i}\ttheta}$ would have a nonvanishing Cantor part, too,
 because the map~$\theta\mapsto e^{i\theta}$ has injective differential
 at any point). Moreover, the abolutely continuous gradient~$\nabla\ttheta_j$
 is bounded in~$L^1(\R^d)$, by Remark~\eqref{remarklifting1}.
 Therefore, the sequence $(\ttheta_j)_{j\in\N}$ is bounded 
 in~$\BV_{\mathrm{loc}}(\R^d)$
 and, up to extraction of a subsequence, it converges weakly in~$\BV$ 
 to a limit~$\ttheta\in\BV_{\mathrm{loc}}(\R^d, \, \R^k)$,
 which is a lifting of~$\u$ and a minimizer of~\eqref{min_lifting}.
\end{remark}

\begin{proof}[Proof of Theorem~\ref{C}]
Let~$\u\in\mathcal{Q}_S$ be given.
We need to prove that~$\mathbb{L}(S) = 2\pi \, \mathbb{P}_{\Z^k,p}(S)$.
 The proof relies on the properties of the jump set for liftings
of torus-valued $W^{1,1}$-maps. In case~$k=1$, this properties
have been studied in, e.g., \cite{GiaquintaModicaSoucek, Ignat-Lifting}.
However, for the reader's convenience, we will present the arguments 
in a (mostly) self-contained way.

\medskip
\noindent
\textbf{Step 1.} First, we prove the inequality
\begin{equation} \label{lifting2-ineq1}
\mathbb{L}_p(S)\geq 2\pi \, \mathbb{P}_{\Z^k,p}(S).
\end{equation}
Let $\ttheta=(\theta_1,\ldots,\theta_k)$ be a ${\BV}$-lifting of $\u$.
We can define a current~$T(\ttheta)$, with coefficients in $\Z^k$,
associated with the jump part of $\ttheta$, in such a way that~$\partial T(\ttheta)=S$. More precisely, $T(\ttheta)$ is defined component-wise
as~$T(\ttheta) = (T_1, \ldots, T_k)$, where 
\[
 T_i =\llbracket \mathrm{S}(\theta_i), \tau_i, c_i \rrbracket
\]
for each index~$i$.
Here, $\mathrm{S}(\theta_i)$ is the jump set of $\theta_i$, 
$c_i(x):=\frac{1}{2\pi}(\theta^{+}_i(x)-\theta^{-}_i(x))$ for $\mathcal{H}^{d-1}$ a.e. $x\in \mathrm{S}(\theta_i)$, and~$\tau_i$ is a unit $(d-1)$-vector orienting the jump set $\mathrm{S}(\theta_i)$. We observe that
the function~$c_i$ must take integer values
(at least~$\mathcal{H}^{d-1}$-almost everywhere),
because~$u_i = e^{i\theta_i}$ is a Sobolev function, with no jump.

We shall now prove that each component~$T_i$ has boundary~$S_i$.
The differential operator~$\d$ satisfies
the identity~$\d^2\theta_i = 0$ in the sense of distributions --- 
that is, for any $v\in  C^{\infty}_{\mathrm{c}}(\R^d, \Lambda_{2}(\R^d))$ one has
\begin{equation}\label{ddoftheta}
 \langle\d\theta_{i}, \, \d^{*}v \rangle=0.
\end{equation}
Here~$\langle\cdot, \, \cdot\rangle$ is the duality
between (form-valued) measures and (vector-valued) functions, 
while~$\d^*$ is the codifferential, defined in~\eqref{adjointoperatorofd}.
From Remark~\ref{remarklifting1}, at the point $x$ where $\theta_{i}$ is approximately differentiable, one has 
\begin{equation}\label{absolutepart}
\d \theta_{i}(x)=u_i(x) \wedge \d u_{i}(x)= j(u_{i})(x).
\end{equation}
In view of~\eqref{absolutepart}, Equation~\eqref{ddoftheta}
can be written as
\begin{equation} \label{ddoftheta2}
 \int_{\R^d}\langle j(u_{i}), \d^{*}v\rangle \, \d x 
 + \int_{S(\theta_i)}(\theta_{i}^{+}-\theta_{i}^{-})(n_i, \d^{*}v) \, \d\mathcal{H}^{d-1} = 0,
\end{equation}
where~$n_i$ is the unit normal to~$\mathrm{S}(\theta_i)$. 
We observe that $(-1)^{d-1}n_i = (\star\tau_i)^\#$, because~$\tau_i$
is the unit tangent $(d-1)$-vector to~$\mathrm{S}(\theta_i)$.
Let~$\omega := \star v$. As~$\star$ is an isometry, we have
\begin{equation} \label{ddoftheta3}
 -(n_i, \d^*v) = (-1)^{d} (n_i, \, \star\d\omega)
 = -(\star n_i, \, \d\omega) = (\tau_i, \, \d\omega) 
\end{equation}
(where~$(\cdot, \, \cdot)$ denotes the Eulidean scalar product
on both vectors and covectors).
From \eqref{adjointoperatorofd}, \eqref{ddoftheta2} and~\eqref{ddoftheta3},
we deduce that
\begin{equation} \label{ddoftheta4}
 \int_{\R^d}\langle \d j(u_{i}), v\rangle \, \d x 
 -2\pi \int_{S(\theta_i)}c_i \, (\tau_i, \, \d\omega) \, \d\mathcal{H}^{d-1} = 0
\end{equation}
or equivalently,
\[
 \star J(u_i) (\omega) - \pi \, T_i(\d\omega) = 0 
 \qquad \Longleftrightarrow \qquad
 \left(\star J(u_i) -\pi\,\partial T_i\right)(\omega) = 0.
\]
As~$v$ is arbitrary, and hence~$\omega$ is, we have proved that
$\pi\,\partial T_i = \frac{1}{\pi} J(u_i) = S_i$, as claimed.
Moreover, by definition of $T(\ttheta)$, one has 
\begin{equation}
\int_{\mathrm{S}(\ttheta)}\|\ttheta^{+}(x)-\ttheta^{-}(x)\|_p \, \d\mathcal{H}^{d-1}(x)=2\pi \, \mathbb{M}_p(T(\ttheta)),
\end{equation}
so~\eqref{lifting2-ineq1} follows.

\medskip
\noindent
\textbf{Step 2.}
Now, we prove the opposite inequality,
\begin{equation} \label{lifting2-ineq2}
\mathbb{L}_p(S)\leq 2\pi \, \mathbb{P}_{\Z^k,p}(S).
\end{equation}
Let $\bar{T}=(\bar{T}_1,\bar{T}_2,\ldots,\bar{T}_k)$ be a $\Z^k$-rectifiable current such that $\partial \bar{T}=S$. 
Let~$\u\in\mathcal{Q}_{S}$ ($\mathcal{Q}_{S}$ is non-empty, see Remark \ref{nonempty}): we shall prove that, for any index~$j$, there exists a ${\BV}$-lifting 
$\bar{\theta}_{j}$ of $u_j$ 
such that the jump part of $\theta_{j}$ is associated with
the current~$\bar{T}_{j}$ --- that is,
\begin{equation} \label{liftingTbar}
 \bar{T}_j=\llbracket S(\bar{\theta}_j), \, \bar{\tau}_j, \,  \bar{c}_j \rrbracket,
\end{equation}
where~$\mathrm{S}(\bar{\theta}_j)$ is the jump set of~$\bar{\theta}_j$, 
$\bar{\tau}_j$ is a unit~$(d-1)$-vector field that orients~$\mathrm{S}(\bar{\theta}_j)$ and $\bar{c}_j(x)=\frac{1}{2\pi}\left(\bar{\theta}_j^{+}(x)-\bar{\theta}_j^{-}(x)\right)$ for $\mathcal{H}^{d-1}$-a.e.~$x\in S_{\bar{\theta}_j}$. 
Indeed, let $\ttheta = (\theta_1, \, \ldots, \, \theta_k)$
be a $\BV$-lifting of $\u$ (which exists, by Theorem~\ref{B}),
and let~$T(\ttheta)$ be 
the current associated with the jump part of $u_j$, 
as constructed in Step~1. As we have seen,
$\partial T(\ttheta)=S_j$, which implies
$\partial (\bar{T} - T(\ttheta)) = 0$.
Therefore, for any index~$j$ there exists a $d$-dimensional 
integral current~$R_j$, with locally finite mass, 
such that $\partial R_j= \bar{T}_j- T_j(\ttheta)$. 
Moreover, $\star R_i$ can be identified 
with a scalar function, which belongs to~$L^{1}_{\rm loc}(\R^d, \, \R)$
and takes its values in~$\Z$, since $R_j$ is integral. From~\eqref{partial-d-bis}, one also has
\begin{equation}\label{jumpset1}
\d(\star R_j)=-\star(\partial R_j)=
\star\left(T_j(\ttheta)-\bar{T}_j\right) \! ,
\end{equation}
\begin{equation}\label{jumpset2}
\bar{\theta}_j=\theta-2\pi \star R_j.
\end{equation}
Then, $\bar{\theta}_j$ is a $\rm BV$-lifting of $u_j$ 
and from~\eqref{jumpset2}, \eqref{jumpset1} we deduce that the current associated with the jump set of~$\bar{\theta}_j$
is $T(\theta_{j})-\left(T_j(\ttheta)-\bar{T}_j\right)=\bar{T}_j$.
(Indeed, by construction $T(\theta_j)$ coincides with the Hodge dual
to the jump part of the distributional differential~$\d\theta_j$, 
up to a sign.) 
Therefore, the function $\bar{\ttheta}=(\bar{\theta}_1,\ldots,\bar{\theta}_k)$ is a ${\BV}$-lifting of $\mathbf{u}=(u_1,\ldots,u_k)$ which satisfies~\eqref{liftingTbar}.
Moreover,
\begin{equation}
\mathbb{L}_p(S) \leq \int_{S(\bar{\ttheta})}\|\bar{\ttheta}^+(x)-\bar{\ttheta}^-(x) \|_p \, \d\mathcal{H}^{d-1}(x)=2\pi\mathbb{M}_p(\bar{T}).
\end{equation}
Therefore, we can conclude that~\eqref{lifting2-ineq2} holds. 
\end{proof}

From Theorem~\ref{C}, we can deduce a sufficient and necessary condition 
for the equality~$\mathbb{H}_p(S) = \mathbb{P}_{\Z^k,p}(S)$ 
to hold, in terms of the lifting constant in~\eqref{Boundlifting0}.

\begin{theorem}\label{D} 
 Let~$1\leq p \leq \infty$ and let~$S$ be a~$\Z^k$-boundary 
 of compact support~$S$ and finite mass
 that satisfies the assumption~\eqref{hp:H}.
 Then, there holds $2\pi \, \mathbb{P}_{\Z^k,p}(S)=\mathbb{H}_{p}(S)$
 if and only if for any $\u\in \mathcal{Q}_{S}$,
 there exists a lifting $\ttheta \in \SBV_{\mathrm{loc}}(\R^d;\R^k)$ 
 of~$\u$ such that 
 \begin{equation}\label{Boundlifting}
  \abs{\ttheta}_{\SBV,p} \leq 2H_{p}(\u).
 \end{equation}
\end{theorem}
\begin{proof}[Proof of Theorem~\ref{D}]
Suppose first that $2\pi \, \mathbb{P}_{\Z^k,p}(S)=\mathbb{H}_{p}(S)$. 
Let~$\u\in\mathcal{Q}_{S}$. By Theorem~\ref{C}
and Remark~\ref{rk:existence-minim-L}, there exists 
a lifting $\ttheta\in\SBV(\R^d;\R^k)$ of~$\u$ such that
\begin{equation}\label{Boundlifting1}
\int_{\mathrm{S}(\ttheta)}\|\ttheta^{+}(x)-\ttheta^{-}(x)\|_p\,\d\mathcal{H}^{d-1}(x)
= \mathbb{L}_p(S) = \mathbb{H}_p(S) \leq H_{p}(\u).
\end{equation}
From Remark~\ref{remarklifting1} and the inequality \eqref{Boundlifting1}, we deduce that
\begin{equation*}\label{Boundlifting2}
\abs{\ttheta}_{\SBV,p} \leq 2H_{p}(\u).
\end{equation*}
Conversely, suppose that for any~$\u\in\mathcal{Q}_{S}$
there exists a lifting $\ttheta \in \SBV(\R^d;\R^k)$ 
of~$u$ such that 
\begin{equation*}\label{Boundlifting3}
\abs{\ttheta}_{\SBV,p} \leq 2H_{p}(\u).
\end{equation*}
Again, from Remark~\ref{remarklifting1}, this implies
\begin{equation} \label{Boundlifting4}
\int_{\mathrm{S}(\ttheta)}\|\ttheta^{+}(x)-\ttheta^{-}(x)\|_p \, \d\mathcal{H}^{d-1}(x)\leq  H_{p}(\u),
\end{equation}
Combining~\eqref{Boundlifting4} with Theorem~\ref{C}, we obtain that
\begin{equation*}
2\pi \, \mathbb{P}_{\Z^k,p}(S)\leq H_{p}(\u).
\end{equation*}
As~$\u$ can be taken arbitrarily, we conclude that
\begin{equation*}
2\pi \, \mathbb{P}_{\Z^k,p}(S) \leq \mathbb{H}_{p}(S).
\qedhere
\end{equation*}
\end{proof}

\begin{remark}\label{RMequi}
The equality~$2\pi \, \mathbb{P}_{\Z^k,p}(S)=\mathbb{H}_{p}(S)$
is true when $k=1$ or $p=1$, because in this case,
we have the lifting property with factor~$2$, 
by Theorem~\ref{B}.
However, the equality~$2\pi \, \mathbb{P}_{\Z^k,p}(S)
=\mathbb{H}_{p}(S)$ may fail in general. One may give examples
(see \cite[Example~4.2]{Bonafini2018} and Remark~\ref{rk:noncalibration} 
above) where
\begin{equation*}
\mathbb{P}_{\Z^k,p}(S)>\mathbb{P}_{\R^k,p}(S).
\end{equation*}
In this case, Theorem \ref{A} implies 
\begin{equation}
\mathbb{H}_{p}(S)<\mathbb{P}_{\Z^k,p}(S).
\end{equation}
and, by Theorem~\ref{D}, the optimal lifting constant
(with respect to the~$p$-norm) for maps~$\u\in\mathcal{Q}_S$
will be strictly larger than~$2$.
\end{remark}

Finally, from Theorem~\ref{A}, Theorem~\ref{B}, and Theorem~\ref{C} we deduce the core result of our paper, namely Theorem \ref{E}.

\begin{proof}[Proof of Theorem \ref{E}]
We need to prove that~$\mathbb{P}_{\Z^k,p}(S) \leq (2 k^{1 - 1/p} - 1) \mathbb{P}_{\R^k,p}(S)$, for any $\Z^k$-boundary~$S$ of compact support.

\medskip
\noindent
\textbf{Step 1.}
We first consider the case~$S$ is an integral polyhedral current
(so that, in particular, $S$ has finite mass and
satisfies the assumption~\eqref{hp:H}).
By Remark~\ref{nonempty}, we know that there exists~$\u\in\mathcal{Q}_S$.
According to Theorem~\ref{B}, 
there exists a lifting~$\ttheta\in\SBV_{\mathrm{loc}}(\R^d, \, \R^k)$ 
of~$\u$ such that
\begin{equation}\label{boundlifting}
\abs{\ttheta}_{\SBV,p} \leq 2k^{1-\frac{1}{p}}H_{p}(\u).
\end{equation}
This inequality, together with Remark~\ref{remarklifting1}, implies
\begin{equation}\label{boundlifting2}
\int_{\mathrm{S}(\ttheta)}\|\ttheta^{+}(x)-\ttheta^{-}(x)\|_p \, \d\mathcal{H}^{d-1}(x)\leq (2k^{1-\frac{1}{p}}-1) H_{p}(\u).
\end{equation}
From~\eqref{boundlifting2},  Theorem~\ref{A} and 
Theorem~\ref{C}, it follows that
\begin{equation}
2\pi \mathbb{P}_{\Z^k,p}(S) = \mathbb{L}(S)
\leq (2k^{1-\frac{1}{p}}-1) \, \mathbb{H}_{p}(S)
\leq (2k^{1-\frac{1}{p}}-1) \, 2\pi \, \mathbb{P}_{\R^k,p}(S)
\end{equation}
and hence
\begin{equation}\label{maininequalityPlateau}
\mathbb{P}_{\Z^k,p}(S)\leq (2k^{1-\frac{1}{p}}-1) \,\mathbb{P}_{\R^k,p}(S),
\end{equation}
as claimed.

\medskip
\noindent
\textbf{Step 2.} Let $S$ be a $\Z^k$-boundary of
dimension~$(d-2)$ in the ambient space $\R^d$. 
According to {\color{blue} \cite[\S~4.2.21]{FeBook},}
for each index~$i=1, \, \ldots, \, k$
there exists a sequence $(S^n_i)_{n\in \N}$
of integer-multiplicity, polyhedral $(d-2)$-cycles 
(that is, currents such that~$\partial S^n_i = 0$ 
--- see Definition \ref{definitionofcurrent}),
as well as sequences of currents~$(Q^{n}_i)_{n\in \N}$,
$(R^{n}_i)_{n\in \N}$ (of dimension~$d-2$, $d-1$ respectively) 
such that
\begin{equation} \label{E1}
S_i-S^n_i= \partial R^{n}_i+  Q^n_i
\end{equation}
and $\mathbb{M}(R^n_i)+\mathbb{M}(Q^n_i) \to 0$
as~$n\to\infty$. Set $S^{n}:=(S^n_1,\ldots,S^n_k)$.
We shall prove that, as~$n\to\infty$,
$\mathbb{P}_{\Z^k,p}(S^n)$, $\mathbb{P}_{\R^k,p}(S^n)$ 
converges to $\mathbb{P}_{\Z^k,p}(S)$, $\mathbb{P}_{\R^k,p}(S)$
respectively. By the previous step,
each~$S^n$ satisfies the inequality~\eqref{maininequalityPlateau}, 
so the theorem will follow by passing to the limit as~$n\to\infty$.
We shall only prove that~$\mathbb{P}_{\R^k,p}(S^n)\to \mathbb{P}_{\R^k,p}(S)$;
the claim $\mathbb{P}_{\Z^k,p}(S^n)\to \mathbb{P}_{\Z^k,p}(S)$
follows by similar arguments.
As both~$S_i$ and~$S^n_i$ are cycles, by taking the differential 
of both sides of~\eqref{E1} we obtain that~$\partial Q^n_i = 0$ 
for any~$n$ and~$i$. Then, by the isoperimetric 
inequality \cite[\S~4.2.10]{FeBook}, for each~$n$ and~$i$
there exists an integral current~$\bar{R}^n_i$ 
such that $\partial \bar{R}^n_i=Q^n_i$ and
\[
 \mathbb{M}(\bar{R}^n_i)\leq C\,\mathbb{M}(Q^n_i)^{\frac{d-1}{d-2}}
\]
for some constant~$C$ that depends on~$d$ only.
Let $C_i^n:=\bar{R}^n_i+R^{n}_i$. Then, we obtain
\begin{equation}
S_i-S^n_i=\partial C^{n}_i, \qquad
\mathbb{M}(C^n_i)\to 0 \quad \textrm{as } n\to\infty.
\end{equation}
Let $T=(T_1,\ldots, T_k)$ be a competitor for the problem $\mathbb{P}_{\R^k,p}(S)$. Define $C^n:=(C^n_1,\ldots,C^n_k)$ and $\bar{T}^n=T-C^n$. 
One has $\partial \bar{T}^n=S^n$ and $\mathbb{M}_p(T-\bar{T}^n)\to 0$
as~$n\to\infty$, because
\begin{equation}
\mathbb{M}_p(T-\bar{T}^n) \leq \sum_{i=1}^k \mathbb{M}(C^n_i) \to 0
\qquad \textrm{as } n\to\infty.
\end{equation}
Therefore, we obtain 
$\limsup_{n\to\infty} \mathbb{P}_{\R^k,p}(S^n)\leq \mathbb{M}_p(T)$
and, since~$T$ is arbitrary,
\[
 \limsup_{n\to\infty} \mathbb{P}_{\R^k,p}(S^n)\leq \mathbb{P}_{\R^k,p}(S).
\]
The opposite inequality, $\liminf_{n\to\infty} \mathbb{P}_{\R^k,p}(S^n)\geq \mathbb{P}_{\R^k,p}(S)$, is obtained by taking an arbitrary
$\R^k$-normal current~$T^n$ such that~$\partial T^n = S^n$ and defining
$T := T^n + C^n$.
\end{proof}


\begin{remark}
	Assume that we have the following equality
	\begin{equation}\label{assume}
	 \mathbb{P}_{\Z^k,p}(S)=\mathbb{H}_{p}(S)
	\end{equation} 
	(which is not always true, as we have seen in Remark \ref{RMequi}).
	Then, Theorem~\ref{A} implies
	\begin{equation*}
	\mathbb{P}_{\Z^k,p}(S)\leq \mathbb{P}_{\R^k,p}(S)
	\end{equation*}
	i.e.
	\begin{equation} \label{equalPlateau}
	\mathbb{P}_{\Z^k,p}(S)=\mathbb{P}_{\R^k,p}(S).
	\end{equation}
	In particular, due to Theorem~\ref{D}, a sufficient condition
	for~\eqref{equalPlateau} is that any~$\u\in \mathcal{Q}_{S}$ admits a lifting $\ttheta \in \SBV_{\mathrm{loc}}(\R^d;\R^k)$
	with~$\abs{\ttheta}_{\SBV,p} \leq 2H_{p}(\u)$.
\end{remark}

\begin{appendix}
\section{Appendix}
In this appendix, we prove some technical results
we used in the paper.
First, we recall some notation from Section~\ref{convolutionmollifier}.
Let~$\rho\in C^{\infty}_{\mathrm{c}}(\R^d)$
be a radial function (i.e.~$\rho(x) = \bar{\rho}(|x|)$ for any~$x\in\R^d$, for some function~$\bar{\rho}\colon\R\to\R$)
such that~$0 \leq \rho \leq 1$, 
$\spt(\rho)\subseteq\bar{\mathbb{B}}^d_1(0)$
and $\int_{\R^d} \rho(x) \, \d x = 1$.
For~$\epsilon > 0$, we define
\begin{equation} \label{mollifier-app}
 \rho_\epsilon(x) := \epsilon^{-d} \rho\left(\frac{x}{\epsilon}\right)
 \qquad \textrm{for } x\in\R^d.
\end{equation}
\begin{lemma} \label{laplacian}
For any~$\epsilon > 0$, there exists an integrable
vector field~$R_{\epsilon} \in C^{\infty}(\R^d \setminus \lbrace 0 \rbrace; \R^d)$ such that
\begin{itemize}
	\item[(i)] $\Div R_{\epsilon}=\delta_{0}-\rho_{\epsilon}$
	in the sense of~$\mathrm{D}^\prime(\R^d)$,
	\item[(ii)] $\spt(R_{\epsilon}) \subset \bar{\mathbb{B}}^d_{\epsilon}(0)$,
	\item[(iii)] $\lVert R_{\epsilon} \rVert_{L_{1}(\R^d)} \to 0$ as $\epsilon \to 0$.
\end{itemize}
\end{lemma}
\begin{proof}
We shall prove that there exists 
an integrable vector field 
$R\in C^{\infty}(\R^d \setminus \lbrace 0 \rbrace; \R^d)$ such that
\begin{equation}\label{divequation}
\Div R=\delta_{0}-\rho
\end{equation}
as distributions in~$\R^d$ and $\spt(R)\subset \bar{\mathbb{B}}^d_{1}(0)$.
Then, the vector field $R_{\epsilon}$ defined by
\begin{equation}
R_{\epsilon}(x):=\epsilon^{-d+1}R\left(\frac{x}{\epsilon}\right)
\end{equation}
will satisfy all the desired properties.

We define
\begin{equation} \label{diverequation-phi}
 \xi (t) := \frac{1}{t^d}\int_{t}^{1}s^{d-1}\rho(s) \, \d s
 \qquad \textrm{for } t > 0
\end{equation}
and $R(x) := \xi(\vert x \vert) \, x$
for~$x\in \R^d \setminus \lbrace 0 \rbrace$.
We observe that, since $\spt\rho\subset\bar{\mathbb{B}}^d_{1}(0)$, 
we have $R(x)=0$ for any~$x\in \R^k$ with~$|x|\geq 1$.
Moreover, for $|x|\leq 1$ and $x\neq 0$ we have
\begin{equation}
\begin{aligned}
|R(x)|&\leq \frac{1}{d|x|^d}(1-|x|^d)|x|
&\leq  \frac{1}{d|x|^{d-1}} .
\end{aligned}
\end{equation}
Therefore, $R\in L^{1}(\R^d;\R^d)$. One has,
\begin{equation}\label{diverquation3}  
\Div R(x)=-\rho(x) \qquad
\textrm{for any } x\in\R^d\setminus\{0\}.
\end{equation}
Indeed, for any $x\neq 0$ we have:
\begin{equation}\label{diverequation4}
\begin{aligned}
\frac{\partial R_i(x)}{\partial x_i} 
&= \frac{\partial}{\partial x_i}\left(\xi(|x|)\right)x_i+\xi(|x|)
=\xi'(|x|)\frac{x^2_i}{|x|} +\xi(|x|).
\end{aligned}
\end{equation}
By differentiating~\eqref{diverequation-phi}, we obtain
\begin{equation} \label{diverequation4.5}
 \xi'(t)=-\frac{\rho(t)}{t}-\frac{d}{t}\xi(t)
\end{equation}
for any~$t > 0$. 
By substituting~\eqref{diverequation4.5} 
into~\eqref{diverequation4}, we obtain
\begin{equation}\label{diverequation5}
\frac{\partial R_i(x)}{\partial x_i} 
=-\frac{x^2_i}{|x|^2}\rho(|x|)-d\frac{x^2_i}{|x|^2}\xi(|x|)+\xi(|x|)
\end{equation}
and, by taking the sum over~$i = 1, \ldots, d$
in both sides of~\eqref{diverequation5}, 
\eqref{diverquation3} follows.
We are left to prove that
\begin{equation} \label{divR}
\Div R=\delta_{0}-\rho
\end{equation}
over $\R^d$ in the sense of distributions, that is,
\begin{equation}
-\int_{\R^d}\nabla \phi(x)\cdot R(x) \, \d x
= \phi(0) - \int_{\R^d}\rho(x) \phi(x) \, \d x
\end{equation}
for any $\phi \in C^{\infty}_{\mathrm{c}}(\R^d;\R)$.
Since $R\in L^{1}(\R^d; \R^d)$, one has
\begin{equation}\label{divequation2}
 \begin{split}
  -\int_{\R^d}\nabla \phi(x)\cdot R(x) \, \d x
  &= -\lim_{\epsilon \to 0} \int_{\R^d\setminus \mathbb{B}^d_{\epsilon}(0)}\nabla \phi(x)\cdot R(x) \, \d x \\
  &= \lim_{\epsilon \to 0} \left(\int_{\R^d\setminus \mathbb{B}^d_{\epsilon}(0)} \Div R(x) \, \phi(x) \, \d x 
  + \int_{\partial \mathbb{B}^d_{\epsilon}(0)}R(x)\cdot\frac{x}{\vert x \vert}\phi(x) \, \d\mathcal{H}^{d-1}(x)\right)
 \end{split}
\end{equation}
by the divergence theorem.
In the right hand side of~\eqref{divequation2},
we have
\begin{equation}\label{divequation2.5}
 \begin{split}
  \int_{\R^d\setminus \mathbb{B}^d_{\epsilon}(0)} \Div R(x) \, \phi(x) \, \d x 
  = -\int_{\R^d\setminus \mathbb{B}^d_{\epsilon}(0)} \rho(x) \, \phi(x) \, \d x 
  \to -\int_{\R^d}\rho(x) \phi(x) \, \d x
 \end{split}
\end{equation}
as $\epsilon \to 0$, because of~\eqref{diverquation3}.
As for the second term, 
we apply the definition of~$R$ to deduce
\begin{equation}
\begin{aligned}
\int_{\partial \mathbb{B}^d_{\epsilon}(0)}
 R(x)\cdot\frac{x}{\vert x \vert}\phi(x) \, \d\mathcal{H}^{d-1}(x)
&= \frac{1}{\epsilon^{d-1}} \int_{\partial \mathbb{B}^d_{\epsilon}(0)} \left(\int_{\epsilon}^1 s^{d-1}\rho(s) \, \d s\right)
 \phi(x) \, \d\mathcal{H}^{d-1}(x) \\
&= \mathcal{H}^{d-1}(\mathbb{S}^{d-1})
 \cdot \int_{\epsilon}^1 s^{d-1}\rho(s) \, \d s
 \cdot 
 \fint_{\partial \mathbb{B}^d_{\epsilon}(0)}\phi(x) \,  \d\mathcal{H}^{d-1}(x),
\end{aligned}
\end{equation}
where~$\fint$ denotes the integral average.
Since~$\rho$ is radial and~$\varphi$ is continuous, we deduce
\begin{equation} \label{divequation6}
 \int_{\partial \mathbb{B}^d_{\epsilon}(0)}
 R(x)\cdot\frac{x}{\vert x \vert}\phi(x) \, \d\mathcal{H}^{d-1}(x)\to \int_{\mathbb{B}^d_{1}(0)}\rho(x) \, \d x \cdot \phi(0) 
 = \phi(0)
\end{equation}
as $\epsilon \to 0$.
Combining~\eqref{divequation2}, \eqref{divequation2.5}
and~\eqref{divequation6}, we obtain~\eqref{divR}.
\end{proof}


In the next lemma, we consider the $p$-mass norm of a 
vector~$v\in\Lambda_{m,\R^{k}}(\R^{d})$, defined 
(as in Equation~\eqref{def_massnorm}) by
\[
 |v|_{\mass,p}:=\sup \left\{ \langle \omega, v \rangle\colon \omega \in \Lambda^{m}_{\R^{k}}(\R^{d}), \ \abs{\omega}_{\comass, p}\leq 1  \right\} \! .
\]
We will also consider the quantity
\begin{equation} \label{N-norm}
N(v):=\inf \left\{ \sum_{i=1}^l \|z_i\|_{p} \, |v_i| \colon v=\sum_{i=1}^l z_i \otimes v_{i}, \, z_i \in \R^k, \, v_{i} \mbox{ is a simple m-vector in } \R^d \right\} \! .
\end{equation}
The following lemma extends~\cite[\S~1.8.1]{FeBook} to the case~$k>1$.

\begin{lemma}\label{characterizationmassnorm}
 Let $v$ be an~$\R^k$-valued $m$-vector in~$\R^d$, 
 i.e. $v\in \Lambda_{m,\R^{k}}(\R^{d})$. Then,
 \begin{equation*}
  \abs{v}_{\mass,p}= N(v).
 \end{equation*}
\end{lemma}
\begin{proof}
Let~$v$ be an $\R^k$-valued $m$-vector in~$\R^d$. We first prove that $|v|_{\mass,p}\leq N(v)$. Suppose that $v=\sum_{i=1}^l z_i \otimes v_{i}$, where
$z_i \in \R^k$ and each~$v_{i}$ is a simple $m$-vector in~$\R^d$.
Let~$\omega \in \Lambda^{m}_{\R^{k}}(\R^{d})$ be such that
$\abs{\omega}_{\comass, p}\leq 1$. Then, one has
\begin{equation}
\begin{aligned}
\langle \omega, v \rangle&=\sum_{i=1}^{l} \langle \omega, z_i \otimes v_{i} \rangle 
\leq \sum_{i=1}^{l} \|z_i\|_{p} \abs{v_i},
\end{aligned}
\end{equation}
where we have made use of the fact that $\abs{\omega}_{\comass, p}\leq 1$ and each $z_i \otimes v_{i}$ has rank~$1$. To prove the reverse inequality, we shall make use of the Hahn-Banach theorem. 
From the definition~\eqref{N-norm}, it is 
not difficult to check that~$N$ is a seminorm
on~$\Lambda_{m, \, \R^k}(\R^d)$.
Then, by the Hahn-Banach theorem, there exists $\omega \in \Lambda^{m}_{\R^{k}}(\R^{d})$ such that
\begin{equation}
\langle \omega, v \rangle=N(v)
\end{equation}
and 
\begin{equation}\label{HahnBanach}
\langle \omega, \tau \rangle \leq N(\tau)
\end{equation}
for any $\tau \in \Lambda_{m, \R^{k}}(\R^{d})$.
We observe that the inequality~\eqref{HahnBanach}
holds true, in particular, for any rank-$1$
vector~$\tau=\theta \otimes t$, where $\theta \in \R^k$ and $t$ is a simple $m$-vector in $\R^d$. Then, by definition of the comass norm 
(see Equation~\eqref{comass}), Equation~\eqref{HahnBanach} implies
\begin{equation}
\abs{\omega}_{\comass, p} \leq 1.
\end{equation}
Therefore,
\begin{equation}
N(v)=\langle \omega, v \rangle 
\leq \abs{v}_{\mass,p} \, \abs{\omega}_{\comass, p} \leq \abs{v}_{\mass,p}
\end{equation}
which concludes the proof.
\end{proof}

\begin{remark}\label{characterizationmasscodimension1}
Let~$\omega$ be a $(\R^k)^*$-valued $1$-covector in $\R^d$.
We define
\[
 \bar{N}(\omega):=\inf \left\{ \sum_{i=1}^l \|z_i\|^*_q \, |\omega_i| \colon \omega=\sum_{i=1}^l z_i \otimes \omega_{i}, \, z_i \in (\R^k)^*, \, \omega_{i} \mbox{ is a 1-covector in } \R^d \right\} \! ,
\]
where~$q$ is the conjugate exponent of~$p$. 
(We recall that~$\norm{\cdot}^*_q$
denotes the dual of the~$q$-norm on~$\R^k$;
therefore, $\norm{\cdot}^*_q$ is the~$p$-norm on~$(\R^k)^*$.)
Then, 
\begin{equation}
\bar{N}(\omega)=N(\star \omega)=\abs{\star \omega}_{\mass,p}.
\end{equation}
To see this, we observe that if 
$\omega=\sum_{i=1}^l z_i \otimes \omega_{i}$, then
$\star \omega=\sum_{i=1}^l z_i \otimes \star \omega_{i}$
(up to identifying each~$z_i$ with its image 
through the isomorphism $(\R^k)^*\to\R^k$ induced by the canonical basis). Moreover, in codimension~$1$, each~$\star \omega_{i}$ is
a $(d-1)$-simple vector and $\abs{\star \omega_{i}}=\abs{\omega_{i}}$.
Therefore, $\bar{N}(\omega)\geq N(\star\omega)$.
The opposite inequality follows by similar arguments.
\end{remark}

\begin{lemma} \label{norm_pisdecreasing} 
	In $\R^k$, the norm $\| \cdot \|_{p}$ is decreasing in $p$
	--- that is, for any $1\leq p_1\leq p_2 \leq \infty$
	and any $z\in \R^k$ one has
	\begin{equation}\label{inequalitynorml_p} 
	\| z \|_{p_2} \leq \| z \|_{p_1}.
	\end{equation}
Moreover, for any $z\in \R^k$, one has
\begin{equation}\label{limitingofnorm}
\lim_{p\to \infty}\| z \|_{p}=\| z \|_{\infty}.
\end{equation}
\end{lemma}
\begin{proof}
	It is easy to see that, when $p_{2}=\infty$, the inequality~\eqref{inequalitynorml_p} holds true. 
	Let us assume that~$p_2<\infty$, and let $z\in \R^k$.
	If $z=0$, \eqref{inequalitynorml_p} is obviously satisfied.
	If $z\neq 0$, let $y_j=\frac{|z_j|}{\|z\|_{p_2}}$
	for any~$j=1, \, \ldots, \, k$
	and set $y=(y_1,\ldots,y_k)$. One has $|y_j|\leq 1$ and, hence,
	\begin{equation}
	|y_{j}|^{p_1}\geq |y_{j}|^{p_2}
	\end{equation}
	for any index~$j$. This implies 
	\begin{equation}\label{holder2}
	\begin{aligned}
	\|y\|_{p_1}\geq 1 
	\qquad \Longleftrightarrow \qquad
	\|z\|_{p_1}\geq \|z\|_{p_2},
	\end{aligned}
	\end{equation}
i.e. the norm $\| \cdot \|_{p}$ is decreasing in $p$. Moreover, by using Holder's inequality we have
	\begin{equation}\label{holder3}
	\| z \|_{p} \leq k^{\frac{1}{p}} \| z \|_{\infty}.
	\end{equation}
Therefore, from~\eqref{holder2} and~\eqref{holder3} we deduce that
\begin{equation}\label{holder4}
\| z \|_{\infty} \leq \| z \|_{p} \leq k^{\frac{1}{p}} \| z \|_{\infty}.
\end{equation}
Let $p$ tends to infinity in \eqref{holder4}, we obtain that
\begin{equation}
\lim_{p\to \infty}\| z \|_{p}=\| z \|_{\infty},
\end{equation}
which is what we needed to prove.
\end{proof}

\begin{lemma}\label{inequalityofmass} 
Let~$T$ be a normal $\R^k$-current 
in~$\R^d$. For any $1\leq p_1 \leq p_2 < \infty$, there holds
\begin{equation}
	\mathbb{M}_{p_2}(T)\leq \mathbb{M}_{p_1}(T)\leq k^{\frac{p_2-p_1}{p_1p_2}}\mathbb{M}_{p_2}(T)
\end{equation}
and for any~$1\leq p_1 \leq \infty$, there holds
\begin{equation}
 \mathbb{M}_{\infty}(T)\leq \mathbb{M}_{p_1}(T)\leq k^{\frac{1}{p_1}}\mathbb{M}_{\infty}(T).
\end{equation}
\end{lemma}
\begin{proof}
For any $1\leq p_1 \leq p_2 < \infty$ and any~$z\in\R^k$,
according to Lemma~\ref{norm_pisdecreasing},
one has
\begin{equation}\label{decreasing}
\| z \|_{p_2} \leq \| z \|_{p_1}.
\end{equation}
On the other hand, by Holder's inequality we have
\begin{equation}\label{holder1}
\| z \|_{p_1} \leq k^{\frac{p_2-p_1}{p_1p_2}} \| z \|_{p_2}.
\end{equation}
Let $\omega \in C^{\infty}_{\mathrm{c}}\left(\R^{d};\Lambda^{d-1}_{\R^{k}}(\R^{d})\right)$. 
From \eqref{decreasing}, using the definition of 
the comass norm (see Equation~\eqref{comass})
and keeping into account that the H\"older conjugate
exponents satisfy~$p_1^* \geq p_2^*$, we deduce
\begin{equation}
\| \omega \|_{\comass, p_1} \leq \| \omega \|_{\comass, p_2},
\end{equation}
This inequality implies, via a duality argument, that
\begin{equation}
\mathbb{M}_{p_2}(T)\leq \mathbb{M}_{p_1}(T).
\end{equation}
In a similar way, from~\eqref{holder1} we deduce
\begin{equation}
\begin{aligned}
\| \omega \|_{\comass, p_2} 
\leq k^{\frac{p^{*}_1-p^{*}_2}{p^{*}_1p^{*}_2}} \, \| \omega \|_{\comass, p_1}
&=k^{\frac{p_2-p_1}{p_1p_2}} \, \| \omega \|_{\comass, p_1}
\end{aligned}
\end{equation}
and hence, by duality,
\begin{equation*}
 \mathbb{M}_{p_1}(T)\leq k^{\frac{p_2-p_1}{p_1p_2}} \, \mathbb{M}_{p_2}(T).
\end{equation*}
When~$p_2 = \infty$, similar arguments apply.
\end{proof}

The next lemma summarizes a few results on 
the comparison between the comass norm of a $(\R^k)^*$-form
(respectively, the mass norm of a~$\R^k$-valued vector)
and the comass (respectively, mass) norm of its components.
Given a (classical) form~$\sigma\in\Lambda^m(\R^d)$ and a 
(classical) vector~$\tau\in\Lambda_m(\R^d)$, we denote
by~$\abs{\sigma}_{\comass}$ the (classical) comass norm of~$\sigma$
and by~$\abs{\tau}_{\mass}$ the (classical) mass norm of~$\tau$.
\begin{lemma} \label{comparenorms} 
 Let~$w\in \Lambda^{m}_{\R^{k}}(\R^{d})$
 and let~$w_{i}=\langle w; \cdot , e_{i} \rangle\in\Lambda^m(\R^d)$
 be the~$i$-th component of~$w$, for~$i=1, \, \ldots, \, k$.
 Let~$p\in [1, \infty]$ and~$q\in [1, \infty]$
 be conjugate exponents, such that~$1/p + 1/q = 1$.
 If~$p\in (1, \infty]$, then there holds
 \begin{equation} \label{comparison-mass1}
  \max_{1\leq i \leq k} \ \abs{w_i}_{\comass}
  \leq \abs{w}_{\comass, p} \leq \left(\sum_{i=1}^k \abs{w_i}_{\comass}^{q}\right)^{\frac{1}{q}} \! ,
 \end{equation}
 while for~$p=1$, there holds
 \begin{equation} \label{comparison-mass2}
  \abs{w}_{\comass, 1} = \max_{1\leq i \leq k} \ \abs{w_i}_{\comass}.
 \end{equation}
 In a similar way, let~$v\in\Lambda_{m,\R^{k}}(\R^{d})$
 be a vector with components~$v_1$, \ldots, $v_k\in\Lambda_m(\R^d)$.
 If~$1 \leq p < \infty$, then there holds
 \begin{equation} \label{comparison-mass4}
 \left(\sum_{i=1}^k \abs{v_i}^p\right)^{\frac{1}{p}} \leq \left(\sum_{i=1}^k \abs{v_i}_{\mass}^p\right)^{\frac{1}{p}}
 \leq \abs{v}_{\mass,p} \leq \sum_{i=1}^k \abs{v_i}_{\mass},
 \end{equation}
 while for~$p = \infty$ we have
 \begin{equation} \label{comparison-mass5}
 \max_{1 \leq i \leq k} \ \abs{v_i} \leq \max_{1 \leq i \leq k} \ \abs{v_i}_{\mass}
 \leq \abs{v}_{\mass,\infty} \leq \sum_{i=1}^k \abs{v_i}_{\mass}.
 \end{equation}
 where~$\abs{\cdot}$ stands for the Euclidean norm. 
\end{lemma}
\begin{remark}
 Equations~\eqref{comparison-mass1} and~\eqref{comparison-mass4} imply that
 \[
  \abs{w}_{\comass,\infty} = \sum_{i=1}^k \abs{w_i}_{\comass}, \qquad
  \abs{v}_{\mass,1} = \sum_{i=1}^k \abs{v_i}_{\mass}
 \]
 for any~$w\in\Lambda^{m}_{\R^{k}}(\R^{d})$
 and any~$v\in\Lambda_{m, \R^k}(\R^d)$.
\end{remark}

\begin{proof}[Proof of Lemma~\ref{comparenorms}]
We recall that~$\norm{\cdot}^*_p$ is the dual norm to~$\norm{\cdot}_p$,
i.e. the~$q$-norm on the dual space~$(\R^k)^*$.

\medskip
\noindent
\textbf{Proof of~\eqref{comparison-mass1}
and~\eqref{comparison-mass2}.}
Let $p>1$. By the definition of comass norm, one has
 \begin{equation}\label{comassappendix}
 \begin{aligned}
 \abs{w}_{\comass, p}&:=\sup \left\lbrace 
  \norm{\langle w ; \tau, \cdot \rangle}^{*}_{p} 
  \colon \vert \tau \vert \leq 1, \ \tau 
  \mbox{ is a simple $m$-vector}\right\rbrace\!\\
 &= \sup \left\lbrace 
  \left(\sum_{i=1}^k \abs{\langle w_i, \tau \rangle}^{q}\right)^{\frac{1}{q}} 
  \colon \vert \tau \vert \leq 1, \ \tau 
  \mbox{ is a simple $m$-vector}\right\rbrace\! \\
 &\leq \left(\sum_{i=1}^k \sup \left\{
  \abs{\langle w_i, \tau \rangle}\colon \vert \tau \vert \leq 1, \ \tau 
  \mbox{ is a simple $m$-vector}
  \right\}^{q} \right)^{\frac{1}{q}}\\
 &=\left(\sum_{i=1}^k \abs{w_i}_{\comass}^{q}\right)^{\frac{1}{q}}.
 \end{aligned}
 \end{equation}
 A similar argument shows that
 \begin{equation}\label{comassappendix2}
 \begin{aligned}
  \abs{w}_{\comass, 1} 
  &\leq \max_{1\leq i \leq k} \abs{w_i}_{\comass}.
 \end{aligned}
 \end{equation}
 Moreover, for each~$i\in\{1, \ldots, k\}$, there exists
 a simple vector~$\tau_i\in\Lambda_m(\R^d)$
 such that~$\abs{\tau_i} = 1$ and~$\abs{w_i}_{\comass} = \langle w_i; \tau_i\rangle$.
 Then, for any index~$i$, we have
 \begin{equation} \label{comassappendix3}
  \begin{aligned}
   \abs{w}_{\comass, p} 
   &= \sup\left\{ \left(\sum_{i=1}^k 
    \abs{\langle w_i, \tau \rangle}^{q}\right)^{\frac{1}{q}} \colon 
    \vert \tau \vert \leq 1, \ \tau \mbox{ is a simple $m$-vector}\right\} \\
   &\geq \left(\sum_{i=1}^k  
    \abs{\langle w_i, \tau_i \rangle}^{q}\right)^{\frac{1}{q}} 
   \geq \abs{\langle w_i, \tau_i \rangle} 
   = \abs {w_i}_{\comass}.
  \end{aligned}
 \end{equation}
 Combining~\eqref{comassappendix3} with~\eqref{comassappendix}
 and~\eqref{comassappendix2}, we deduce~\eqref{comparison-mass1}
 and~\eqref{comparison-mass2}.


 \medskip
 \noindent
 \textbf{Proof of~\eqref{comparison-mass4}.}
 Let us take a vector~$v\in\Lambda_{m,\R^{k}}(\R^{d})$
 with components~$v_1$, \ldots, $v_k\in\Lambda_m(\R^d)$.
 Let $p \in [1, \, \infty]$. 
 We have
 \begin{equation}\label{comassappendix4}
  \begin{aligned}
   \abs{v}_{\mass,p}
   &:= \sup \left\lbrace \langle w, v \rangle\colon 
    w \in \Lambda^{m}_{\R^k}(\R^d), \ \abs{w}_{\comass, p} \leq 1\right\rbrace\!\\
   &\geq  \sup \left\lbrace  \sum_{i=1}^k \langle w_i, v_i \rangle\colon
    \ w \in \Lambda^{m}_{\R^k}(\R^d), \ \|w^p\|_{q} \leq 1\right\rbrace\!,
  \end{aligned}
 \end{equation}
 where $w^p :=(\abs{w_1}_{\comass},\ldots,\abs{w_k}_{\comass})\in\R^k$.
 In~\eqref{comassappendix4}, we have made use of the inequalities~\eqref{comparison-mass1} and~\eqref{comparison-mass2}.
 By the Hahn-Banach Theorem, for each $i=1,\ldots,k$ 
 there exists~$w_{i}\in \Lambda^{m}(\R^d)$ such that
 $\abs{w_{i}}_{\comass} \leq 1$ and
 \begin{equation}
  \langle w_{i}, v_i \rangle=\abs{v_i}_{\mass}.
 \end{equation} 
 Now, suppose for a moment that~$p < \infty$.
 Let $\bar{w}_{i}:=\lambda \abs{v_{i}}^{p-1}_{\mass}w_{i}$, 
 where $\lambda:=\left(\sum_{j=1}^k \abs{v_j}_{\mass} ^{p}\right)^{\frac{1}{p}-1}$. Then,
 it is not difficult to check that the form~$\bar{w} = (\bar{w}_1, \ldots, \bar{w}_k)$ satisfies~$\|\bar{w}^p\|_{q} \leq 1$. 
 Moreover, the inequality~\eqref{comassappendix4} implies
 \begin{equation}
  \langle \bar{w}, v \rangle
   =\left(\sum_{j=1}^k \abs{v_j}_{\mass}^p\right)^{\frac{1}{p}}
   \leq \abs{v}_{\mass,p}.
 \end{equation}
 Moreover, by the definition of mass norm, 
 we obtain that for any $i=1,\ldots,k$
 \begin{equation} \label{blarg}
  \abs{v_i}\leq \abs{v_i}_{\mass}.
 \end{equation}
 This implies that
 \begin{equation}
  \left(\sum_{i=1}^k \abs{v_i}^p\right)^{\frac{1}{p}} \leq \left(\sum_{i=1}^k \abs{v_i}_{\mass}^p\right)^{\frac{1}{p}}.
 \end{equation}
 In case~$p=\infty$ we proceed in a similar way.
 We fix an index~$i\in\{1, \ldots, k\}$,
 define~$\bar{w}_i := w_i$, $\bar{w}_j := 0$
 for~$i\neq j$, and~$\bar{w} := (\bar{w}_1, \ldots, \bar{w}_k)$.
 From~\eqref{comassappendix4} and~\eqref{blarg}, we deduce
 \begin{equation} 
   \abs{v_i} \leq \abs{v_i}_{\mass} = \langle \bar{w}, v\rangle 
   \leq \abs{v}_{\mass,p} \qquad \textrm{for any }
   i\in\{1, \ldots, k\}.
 \end{equation}
 In order to complete the proof of~\eqref{comparison-mass4} 
 and~\eqref{comparison-mass5}, it only remains to prove the 
 following inequality:
 \begin{equation} \label{blargbis}
  \abs{v}_{\mass,p} \leq \sum_{i=1}^k \abs{v_i}_{\mass}
 \end{equation}
 for any~$p\in [1, \, \infty]$.
 We observe that for any $w=(w_1,\ldots,w_k) \in \Lambda^{m}_{\R^{k}}(\R^{d})$ such that $\abs{w}_{\comass, p} \leq 1$, there holds
 \begin{equation}
  \begin{aligned}
   \langle w, v \rangle 
   &= \sum_{i=1}^k \langle w_i, v_i \rangle
   \leq \sum_{i=1}^k \abs{v_i}_{\mass},
  \end{aligned}
 \end{equation}
 because of~\eqref{comassappendix3}. 
 By taking the supremum over~$w$, \eqref{blargbis} follows.
\end{proof}

\end{appendix}


\begin{small}
 \bigskip
 \noindent
 \textsc{G. Canevari,} \\
 Dipartimento di Informatica, Universit\`a di Verona,\\
 Strada le Grazie~15, 37134 Verona, Italy.\\
 \textit{E-mail address:} giacomo.canevari@univr.it
 
 \bigskip
 \noindent
 \textsc{V.P.C. Le},  \footnote{This work was initiated and a large part of the project was carried out when the second author was a postdoc at the Dipartimento di Informatica, Università di Verona, Italy.} 
 \\
 Institut für Mathematik, Universität Heidelberg,\\
 Im Neuenheimer Feld 205, 69120 Heidelberg, Germany.\\
 \textit{E-mail address:} cuong.le@uni-heidelberg.de

\end{small}

\end{document}